\documentclass[a4paper,12pt]{article} 

\usepackage[utf8]{inputenc}
\usepackage{xspace}
\usepackage{mathrsfs} 
\usepackage[english]{babel}
\usepackage{dirtytalk}
\usepackage{physics}
\usepackage{csquotes}
\usepackage{amssymb}
\usepackage{amsmath}
\usepackage{amsfonts}
\usepackage{amsthm}
\usepackage{stackengine}
\usepackage{graphicx}
\usepackage[dvipsnames]{xcolor}
\usepackage{subcaption}
\usepackage[export]{adjustbox}
\usepackage{import}
\usepackage{algorithm}
\usepackage{algpseudocode}
\usepackage{parskip}
\usepackage{caption}   
\PassOptionsToPackage{%
backend=biber, bibencoding=utf8 
}{biblatex}
\usepackage{biblatex}
\AtBeginBibliography{\small}

\definecolor{newpink}{HTML}{fc309f}

\usepackage{hyperref}
    \hypersetup{
    unicode=true,                 
    pdftoolbar=true,              
    pdfmenubar=true,              
    pdffitwindow=true,            
    pdfnewwindow=true,            
    pdfkeywords={},               
    colorlinks=true,              
    linkcolor=magenta,                
    citecolor=magenta,                
    urlcolor=magenta
}
\usepackage{cleveref}
\usepackage{placeins}

\usepackage[labelfont=bf]{caption}
\usepackage{sectsty}

\sectionfont{\large\sffamily}     
\subsectionfont{\normalsize\sffamily}  
\subsubsectionfont{\normalsize\sffamily}

\usepackage[a4paper,top=1.5cm,bottom=2cm,outer=2cm,inner=2cm,includeheadfoot]{geometry}
\usepackage{fancyhdr}
\newtheoremstyle{break}
  {\topsep}{\topsep}%
  {\itshape}{}%
  {\bfseries}{}%
  {\newline}{}%
\theoremstyle{break}
\newtheorem{theorem}{Theorem}[section]

\newtheorem{lemma}[theorem]{Lemma}
\newtheorem*{remark}{Remark}
\newtheorem{definition}[theorem]{Definition}

\theoremstyle{definition}
\newtheorem{example}{Example}

\newcommand{\R}{\mathbb{R}}
\newcommand{\threespace}{\,\,\,}

\newcommand{\duality}[2]{ \langle #1 , #2 \rangle }
\newcommand{\jump}[1]{[\![ \nabla #1 ]\!]} 
\newcommand{\fem}[1]{#1_h} 

\newcommand{\HH}{\ensuremath{\operatorname H}\xspace}
\newcommand{\LL}{\ensuremath{\operatorname L}\xspace}

\newcommand{\leb}[1]{\ensuremath{\LL^{#1}}}

\newcommand{\sobh}[1]{\ensuremath{\HH^{#1}}}
\newcommand{\sobhO}{\ensuremath{\HH^{1}_{0}}}

\renewcommand{\vec}[1]{\ensuremath{\boldsymbol{#1}}}
\newcommand{\qp}[1]{\ensuremath{\!\left({#1}\right)}}
\newcommand{\Norm}[1]{\ensuremath{\left\|#1\right\|}}

\newcommand{\theTitle}{Adaptive Regularisation for PDE-Constrained Optimal Control}

\newcommand{\theShortAuthor}{Power \& Pryer}

\providecommand{\keywords}[1]
{
  \small	
  \textbf{Keywords:} #1
  \normalsize
}

\begin{document}
\fancyhead[R]{\theShortAuthor}
\fancyhead[C]{}
\fancyhead[L]{\theTitle}
\fancyfoot[C]{\thepage}
\parskip=12pt

\thispagestyle{plain}

\title{Adaptive Regularisation for PDE-Constrained Optimal Control}

\makeatletter
\begin{center}

{\sffamily\Large\selectfont\bfseries\@title\par}

{\large Jenny Power\textsuperscript{a,$*$}, Tristan Pryer\textsuperscript{a,b} \par}  
        \vspace{0.5em}  
        
        {\small 
        \textsuperscript{a}Department of Mathematical Sciences, University of Bath, Bath, BA27AY, United Kingdom \par
        \textsuperscript{b}Institute for Mathematical Innovation, Bath, BA27AY, United Kingdom 
        }
\end{center}
\makeatother

\footnotetext{\textsuperscript{$*$}Corresponding author. \par Email addresses: \texttt{jip30@bath.ac.uk} (Jenny Power), \texttt{tmp38@bath.ac.uk} (Tristan Pryer)}

\rule{\textwidth}{0.1pt}
\vspace{-2em}
\section*{Abstract}
PDE constrained optimal control problems require regularisation to
ensure well-posedness, introducing small perturbations that make the
solutions challenging to approximate accurately. We propose a finite
element approach that couples both regularisation and discretisation
adaptivity, varying both the regularisation parameter and mesh-size
locally based on rigorous a posteriori error estimates aiming to
dynamically balance induced regularisation and discretisation errors,
offering a robust and efficient method for solving these problems.  We
demonstrate the efficacy of our analysis with several numerical
experiments.    

\keywords{
PDE-constrained optimal control,
Adaptive regularisation,
Finite element methods,
A posteriori error estimates
} \\
\rule{\textwidth}{0.1pt}

\section{Introduction}
Optimal control problems constrained by partial differential equations
(PDEs) are fundamental in various scientific and engineering
applications. These problems play an important role in areas such as
fluid dynamics \cite{DolgovStoll:2017}, heat
conduction \cite{HarbrechtKunothSimonciniUrban:2023}, structural
optimisation \cite{KolvenbachLassUlbrich:2018}, crystal
growth \cite{ChunHesthaven:2008}, and medical treatment planning, such
as radiotherapy \cite{CoxHattamKyprianouPryer:2023}. Solving these
problems efficiently and accurately is essential for making
advancements in these domains.

Traditionally, the mathematical framework for PDE-constrained optimal
control problems, particularly those involving linear and nonlinear
elliptic or parabolic PDEs, has been
well-established \cite{BieglerGhattasHeinkenschlossBloemen-Waanders:2003,ReesDollarWathen:2010,De-los-Reyes:2015}. These
formulations typically lead to first-order optimality conditions,
where adjoint states or Lagrange multipliers are introduced to derive
the necessary conditions for optimality. Such formulations have paved
the way for numerous numerical methods, providing a robust theoretical
foundation for solving these problems.

Common numerical approaches for solving these problems involve
discretising the continuous optimal control problem using finite
element methods (FEM) \cite{AllendesFuicaOtarola:2020,
BrennerGedickeSung:2018}, spectral
methods \cite{RodenMills-WilliamsPearsonGoddard:2022}, or wavelet
methods \cite{Kunoth:2005}. Among these, FEM is widely favoured for
its flexibility in handling complex geometries and boundary
conditions, making it well-suited for problems involving non-smooth
solutions or intricate domains. However, accurately solving such
PDE-constrained control problems often comes with high computational
cost, particularly when dealing with large-scale or high-dimensional
systems.

A significant aspect of PDE-constrained optimisation is the choice of
the regularisation, which is typically introduced to ensure
well-posedness and stability of the problem. However, choosing an
appropriate regularisation parameter can be challenging, especially
when dealing with non-smooth or complex solutions. While much of the
existing work has focused on a priori analysis of FEMs \cite{langer2024adaptive},
less attention has been given to adaptive strategies that adjust the
regularisation parameter dynamically based on the solution's behaviour,
particularly in the context of a posteriori error analysis.

Adaptive regularisation strategies have been widely studied in the
context of inverse problems \cite{van2017learning, hong2017adaptive}, where balancing
regularisation and data fidelity is important. Given the close
relationship between inverse problems and PDE-constrained optimal
control problems \cite{clason2020optimal}, techniques from adaptive inverse problem
regularisation provide a valuable perspective for designing robust
numerical methods in control applications.

This paper presents a novel approach that addresses this gap by
introducing an adaptive strategy for regularising PDE-constrained
optimal control problems. We propose a method that leverages a
posteriori error estimates to adaptively vary the regularisation
parameter across the computational domain. This allows the
regularisation to be varied elementwise, balancing the approximation
error and regularisation inconsistency. This potentially enhances
solution accuracy whilst also reducing computational overhead by
focusing computational effort where it is most needed, enabling more
efficient handling of non-smooth solutions, boundary layers and
complex geometries.

The novelty of our approach lies in the dynamic adjustment of both the
regularisation parameter and the mesh size. This is guided by rigorous
error estimators derived from a posteriori estimates, extending
traditional adaptive finite element
methods \cite{verfurth2013posteriori, ainsworth2000posteriori}. Our
method ensures that the regularisation and discretisation errors are
balanced, leading to more efficient and accurate solutions for
nonlinear optimisation problems.

The remainder of this paper is organised as follows:
In section \ref{sec:2}, we introduce the linear quadratic optimal control
problem that forms the basis of this work. We derive the optimality
conditions and prove well-posedness, showing the equivalence between
solving the optimal control problem and a fourth-order singularly
perturbed equation. In section \ref{sec:3}, we describe the finite element
discretisation of the problem. In section \ref{sec:4}, we
derive global upper and local lower error bounds, which are then used
to develop an adaptive numerical scheme. Finally, in section \ref{sec:results}, we
demonstrate the algorithm's convergence and present numerical
results that showcase the effectiveness of our approach.

\section{Optimal Control Problem} \label{sec:2}
Let $\Omega \subset
\mathbb{R}^n$ be an open bounded Lipschitz domain with polyhedral boundary $\partial \Omega$.
For $D \subseteq \mathbb{R}^n$, where $n = 1, 2, 3$, we denote the
standard Lebesgue spaces by $\leb{p}(D)$ for $1 \leq p \leq \infty$ and Hilbert spaces $\sobh{m}(D)$ for $0 \leq m \leq \infty$.
The $\leb{2}$ inner product over
$D$ is denoted $\langle \cdot, \cdot \rangle_D$. 
Additionally, when $D = \Omega$, we omit the
subscript and write $\|\cdot\|$ and $\langle \cdot, \cdot \rangle$ to
represent the $\leb{2}(\Omega)$-norm and inner product respectively.
We let $\sobhO(D)$ be
functions in $\sobh1(D)$ with vanishing trace.

As a model problem, we consider a linear quadratic optimal control
problem (OCP) posed on $\Omega$. Let $u \in
\sobhO(\Omega)$ represent the state variable and $f \in
\leb{2}(\Omega)$ the control variable. We seek to find the optimal
state-control pair $(u, f) \in \sobhO(\Omega) \times
\leb{2}(\Omega)$ such that

\begin{equation} \label{eq:costfunctional}
  (u, f)
  =
  \arg\min_{\qp{v, g}\in \sobhO(\Omega)\times\leb{2}(\Omega)} \left( \frac{1}{2} \| v - d \|^2
  +
  \frac{\alpha}{2} \| g \|^2 \right),
\end{equation}
subject to the constraint
\begin{equation} \label{eq:PDEconstraint}
     a(v, \varphi) := \duality{\nabla v}{\nabla \varphi} = \duality{g}{\varphi}, \qquad \forall \varphi \in \sobhO(\Omega),
\end{equation}
where $d \in \leb{2}(\Omega)$ is the desired state and $\alpha \in \mathbb{R}_+$ is the regularisation
parameter, typically with $\alpha \ll 1$.

\subsection{Optimality Conditions}

Since this is a linear quadratic control problem, the optimal solution
can be obtained by solving the first-order optimality conditions,
commonly referred to as the Karush–Kuhn–Tucker (KKT) conditions
\cite{Antil2018}. For the problem described above, these conditions
are given weakly by: Find
$\qp{u,f,z} \in \sobhO(\Omega) \times \leb2(\Omega) \times \sobhO(\Omega)$
such that
\begin{align}
  \duality{\nabla u}{\nabla \varphi} &= \duality{f}{\varphi}, & & \forall \varphi \in \sobhO(\Omega), \label{eq:KKT1} \\
  \duality{\nabla z}{\nabla \psi} + \duality{u}{\psi} &= \duality{d}{\psi}, & & \forall \psi \in \sobhO(\Omega), \label{eq:KKT2} \\
  \alpha \duality{f}{\tau} - \duality{z}{\tau} &= 0, & & \forall \tau \in \leb{2}(\Omega), \label{eq:KKT3}
\end{align}
where $z$ is a Lagrange multiplier, referred to as the adjoint
variable. It can be shown that the solution $(u, f, z)$ to this problem
coincides with the optimal state and control given
by \eqref{eq:costfunctional}--\eqref{eq:PDEconstraint}.

\begin{remark}[Equivalence of OCP to a singularly perturbed equation]
  \label{thm:equivOCP1}
  The optimal control problem presented in \eqref{eq:costfunctional}
  and \eqref{eq:PDEconstraint} is equivalent to solving the singularly
  perturbed biharmonic equation
  \begin{equation*} \label{eq:biharmonic}
    \alpha \Delta^2 u + u = d \quad \text{in } \Omega,
  \end{equation*}
  with the boundary conditions
  \begin{equation*} \label{eq:biharmbcs}
    u = \Delta u = 0 \quad \text{on } \partial \Omega.
  \end{equation*}  
\end{remark}

\subsection{Primal-Dual Formulation}

It is convenient to reformulate equations
\eqref{eq:KKT1}--\eqref{eq:KKT3}
for the state variable $u$ and the adjoint variable $z$. We seek $u, z
\in \sobhO(\Omega)$ such that
\begin{align}
    \duality{\alpha \nabla u}{\nabla \varphi} - \duality{z}{\varphi} &= 0, & & \forall \varphi \in \sobhO(\Omega), \label{eq:pd1} \\
    \duality{\nabla z}{\nabla \psi} + \duality{u}{\psi} &= \duality{d}{\psi}, & & \forall \psi \in \sobhO(\Omega). \label{eq:pd2}
\end{align}

Let $X := \sobhO(\Omega) \times \sobhO(\Omega)$. We define the
bilinear form
\begin{align} 
  b((u, z), (\varphi, \psi)) := \duality{\alpha \nabla u}{\nabla \varphi} - \duality{z}{\varphi} + \duality{\nabla z}{\nabla \psi} + \duality{u}{\psi},
  \label{eq:b}
\end{align}
which allows us to reformulate \eqref{eq:pd1}--\eqref{eq:pd2} into the
following problem: Find $(u, z) \in X$ such that
\begin{align} 
  b((u, z), (\varphi, \psi)) = \duality{d}{\psi} =: L((\varphi, \psi)), \quad \forall (\varphi, \psi) \in X.
  \label{eq:bform}
\end{align}

The energy norm associated with $b$ is given by
\begin{align} \label{eq:xnorm}
    \|(u, z)\|_{X} := \left( \alpha \|\nabla u \|^2 + \|\nabla z\|^2 \right)^{\frac{1}{2}}.
\end{align}


\begin{theorem}[Well-posedness] \label{thm:wp1}
  Let $d \in \leb{2}(\Omega)$ be given. Then, there exists a unique
  pair $(u, z) \in X$ that solves \eqref{eq:bform}.
\end{theorem}

\begin{proof}
  We show that the bilinear form defined in \eqref{eq:b} is coercive
  and bounded. To begin, note
  \begin{align*}
    b((u, z), (u, z)) &= \alpha \|\nabla u\|^2 + \|\nabla z\|^2 \\
    &= \|(u, z)\|_{X}^2, 
  \end{align*}
  hence the problem is coercive. For boundedness apply the
  Cauchy-Schwarz inequality to \eqref{eq:b} and use the Poincar\'e
  inequality to see
  \begin{align*} 
    b((u, z), (\varphi, \psi)) &\leq C_B \|(u, z)\|_{X} \|(\varphi, \psi)\|_{X},
  \end{align*}
  where
  \begin{equation*} 
    C_B = \left(1 + C_p^2 \max \left(1, \alpha^{-1} \right)\right),
  \end{equation*}
  and $C_p$ is the Poincar\'e constant.
  
  Since the bilinear form \eqref{eq:b} is both coercive and bounded,
  the Lax-Milgram theorem guarantees the existence of a unique solution to
  \eqref{eq:bform}.
\end{proof}

\section{Finite Element Discretisation}  \label{sec:3}

Let $\mathcal{T}$ be a conforming, shape-regular partition of $\Omega$ into simplices (or quadrilaterals/hexahedra) $T$. The diameter of an element $T$ is denoted by $h_T$, and $h = \max\{ h_T : T \in \mathcal{T} \}$. We let $\mathcal{E}$ denote the set of  
$(n-1)$-dimensional facets (edges in 2D, faces in 3D)
 associated with the partition $\mathcal{T}$ and let $h_E$ denote the size of a facet $E$.

For each element $T \in \mathcal{T}$, we define the local finite element space $\mathcal{R}_k(T)$ as
\begin{equation*}
  \mathcal{R}_k(T) = \left\{ \begin{array}{cl}
    \mathbb{P}_k(T) & \textrm{if } T \textrm{ is a simplex}, \\
    \mathbb{Q}_k(T) & \textrm{if } T \textrm{ is an affine quadrilateral/hexahedron},
  \end{array}\right.
\end{equation*}
where $\mathbb{P}_k(T)$ denotes the polynomials of total degree $k$ on $T$, and $\mathbb{Q}_k(T)$ denotes the polynomials of degree less than or equal to $k$ in each variable, mapped from a reference element through an affine mapping.

Now, for $k \geq 1$, we define the global finite element space as
\begin{equation*}
  V_h := \{ v_h \in C^0(\Omega) : v_h|_T \in \mathcal{R}_k(T) \; \forall \, T \in \mathcal{T} \} \cap \sobhO(\Omega).
\end{equation*}

We denote the patch of an element $T$ by $\widetilde{T}$, where
\begin{equation*}
    \widetilde{T} := \{ S \in \mathcal{T} : \overline{T} \cap \overline{S} \neq \emptyset \}.
\end{equation*}

We also denote the pair of elements sharing a facet $E$ by $\widehat{T}:= T^{+} \cup T^{-}$. We define the jump operator of a vector quantity $\vec{v}$ over a facet by
\begin{equation*} 
  [\![ \vec{v} ]\!]  = \vec{v} \rvert_{T^+} \cdot \vec{n}_{T^+} + \vec{v} \rvert_{T^-} \cdot \vec{n}_{T^-} .
\end{equation*}

Finally, we define the space $X_h = (V_h, V_h) \subset X$.

We adopt an ``optimise then discretise" approach and approximate
\eqref{eq:bform} over the finite element space. Additionally, we
replace the constant parameter $\alpha$ with a piecewise function,
allowing its value to vary elementwise. This is the heart of our
adaptive regularisation strategy. Let $(u_h, z_h) \in X_h$ be the
finite element approximation of $(u, z)$. We introduce the elementwise
discontinuous function $\alpha_h$, which we take to be piecewise constant
for exposition, where $\alpha_h |_T \in \mathbb{R}$ and $\alpha \leq \alpha_h |_T \,\leq 1$ for each element $T$. We define the bilinear form
\begin{align} 
  b_h((\fem{u}, \fem{z}), (\varphi_h, \psi_h))
  := \duality{\alpha_h \nabla \fem{u}}{\nabla \varphi_h}
  - \duality{\fem{z}}{\varphi_h}
  + \duality{\nabla \fem{z}}{\nabla \psi_h}
  + \duality{\fem{u}}{\psi_h}. \label{eq:bh}
\end{align}

We then seek a pair $(u_h, z_h) \in X_h$ such that
\begin{align} 
    b_h((\fem{u}, \fem{z}), (\varphi_h, \psi_h)) =  L((\varphi_h, \psi_h)), \quad \forall (\varphi_h, \psi_h) \in X_h. \label{eq:bhform}
\end{align}

The numerical scheme naturally induces an energy norm,
\begin{align} 
  \Norm{(u_h, z_h)}_{X_h}
  :=
  \left(
    \Norm{\alpha_h^{\frac{1}{2}} \nabla u_h}^2
    +
    \Norm{\nabla z_h}^2
  \right)^{\frac{1}{2}}. \label{eq:Xhnorm}
\end{align}

\begin{lemma}[Well-posedness of the FE approximation] \label{lem:wpd}
  Let $d \in \leb{2}(\Omega)$ be given data. Then, there exists a unique pair
  $(u_h, z_h) \in X_h$ that solves \eqref{eq:bhform}.
\end{lemma}

\begin{proof}
  The proof follows the same steps as in Theorem \ref{thm:wp1}.
  The bilinear form $b_h$ \eqref{eq:bh} is coercive and bounded with respect to the norm $\Norm{\cdot}_{X_h}$ defined in \eqref{eq:Xhnorm}. The boundedness constant $C_B$ is the same as on the continuum case,
  noting that $\alpha_h \geq \alpha > 0$.
\end{proof}

\begin{theorem}[Galerkin Orthogonality] \label{thm:orthogonality}
Let $(u, z)$ solve \eqref{eq:bform} and $(u_h, z_h)$ solve \eqref{eq:bhform}. Then,
\begin{align} \label{eq:orthogonality}
    b\big((u - \fem{u}, z - \fem{z}), (\varphi_h, \psi_h)\big) + \duality{(\alpha - \alpha_h)\nabla \fem{u}}{\nabla \varphi_h} = 0, \quad \forall (\varphi_h, \psi_h) \in X_h.
\end{align}
\end{theorem}

\begin{proof}
  Since $X_h \subset X$, we have
  \begin{equation*}
    b((u, z), (\varphi_h, \psi_h)) = L((\varphi_h, \psi_h)), \quad \forall (\varphi_h, \psi_h) \in X_h.
  \end{equation*}
  Subtracting \eqref{eq:bhform} from this expression, we obtain
  \begin{align*} 
    b((u, z), (\varphi_h, \psi_h)) - b_h((\fem{u}, \fem{z}), (\varphi_h, \psi_h)) = 0, \quad \forall (\varphi_h, \psi_h) \in X_h.
  \end{align*}
  Noting that
  \begin{equation*}
    b_h((\fem{u}, \fem{z}), (\varphi_h, \psi_h)) = b((\fem{u}, \fem{z}), (\varphi_h, \psi_h)) - \duality{(\alpha - \alpha_h)\nabla \fem{u}}{\nabla \varphi_h}
  \end{equation*}
  concludes the proof.
\end{proof}

Let the error between the continuous solution and the finite element
approximation be denoted by $\vec{e} := (e_u, e_z) = (u - \fem{u}, z -
\fem{z})$.

\begin{lemma}[Strang's Lemma] \label{lem:strang}
  Let $(u, z) \in X$ solve \eqref{eq:bform}, and let $(u_h, z_h)$ be
  the finite element approximation solving \eqref{eq:bhform}. Then,
  there exists a constant $C > 0$ such that
  \begin{align}
    \norm{\vec{e}}_{X_h}
    \leq
    C
    \left(
      \inf_{(v_h, w_h) \in X_h}
      \| (u - v_h, z - w_h) \|_{X_h}
      +
      \norm{\alpha_h^{-\frac{1}{2}}(\alpha_h - \alpha)\nabla u}
    \right).
  \end{align}
\end{lemma}
\begin{proof}
  To begin, note that for any $(v_h, w_h) \in X_h$, we have
  \begin{align*}
    \Norm{(u_h - v_h, z_h - w_h)}_{X_h}^2
      &=
      b_h\left(
        (u_h - v_h, z_h - w_h),
        (u_h - v_h, z_h - w_h)
      \right) \\
    \begin{split} 
      &=
      b_h\left(
        (u_h - u, z_h - z),
        (u_h - v_h, z_h - w_h)
      \right) \\
      &\quad +
      b_h\left(
        (u - v_h, z - w_h),
        (u_h - v_h, z_h - w_h)
      \right).
    \end{split}
  \end{align*}
  Using the definitions of the bilinear forms \eqref{eq:b} and \eqref{eq:bh}, we obtain
  \begin{align} 
    \begin{split}
      \Norm{(u_h - v_h, z_h - w_h)}_{X_h}^2
      =
      &\duality{(\alpha_h - \alpha)\nabla u}{\nabla (u_h - v_h)}         \\ & 
      +
      b_h\left(
        (u - v_h, z - w_h),
        (u_h - v_h, z_h - w_h)
      \right).
    \end{split}
    \label{eq:temp}
  \end{align}
  Applying the Cauchy-Schwarz inequality to the first term, we get
  \begin{align*}
      \duality{(\alpha_h - \alpha)\nabla u}{\nabla (u_h - v_h)}         
      &=
      \duality{\alpha_h^{-\frac{1}{2}}(\alpha_h - \alpha)\nabla u}{\alpha_h^{\frac{1}{2}} \nabla (u_h - v_h)}         
      \\
      &\leq
      \Norm{\alpha_h^{-\frac{1}{2}}(\alpha_h - \alpha)\nabla u}
      \Norm{\alpha_h^{\frac{1}{2}}\nabla (u_h - v_h)}.
  \end{align*}
  By definition of $\Norm{\cdot}_{X_h}$ \eqref{eq:Xhnorm}
  \begin{equation*}
    \duality{(\alpha_h - \alpha)\nabla u}{\nabla (u_h - v_h)} \leq \Norm{\alpha_h^{-\frac{1}{2}}(\alpha_h - \alpha)\nabla u} \Norm{(u_h - v_h, z_h - w_h)}_{X_h}.
  \end{equation*}
  Using boundedness of $b_h$ from Lemma \ref{lem:wpd}, we obtain from \eqref{eq:temp}
  \begin{align}
    \Norm{(u_h - v_h, z_h - w_h)}_{X_h} \leq C_B \Norm{(u - v_h, z - w_h)}_{X_h} + \Norm{\alpha_h^{-\frac{1}{2}}(\alpha_h - \alpha)\nabla u}.
  \end{align}
  Making use of the triangle inequality 
   \begin{align}
    \Norm{u-u_h, z-z_h}_{X_h} \leq \Norm{u-v_h, z-w_h}_{X_h} + \Norm{u_h-v_h, z_h-w_h}_{X_h}
   \end{align}
  and taking the infimum concludes the proof.
\end{proof}

\begin{remark}[{A priori coupling}]
  Lemma \ref{lem:strang} suggests an appropriate a priori coupling between \( \alpha_h \) and the mesh size $h$. Indeed, for $u, z \in H^{r+1}(\Omega)$,
  \begin{equation*}
      \inf_{(v_h, w_h) \in X_h}
      \| (u - v_h, z - w_h) \|_{X_h}
      \leq
      C h^{\min(k,r)} \left( \Norm{u}_{H^{r+1}(\Omega)} + \Norm{z}_{H^{r+1}(\Omega)} \right),
  \end{equation*}
  for a mesh-size-independent constant $C$. 
  The error in the inconsistency term is controlled by
  \begin{equation*}
    \norm{\alpha_h^{-\frac{1}{2}}(\alpha_h - \alpha)\nabla u} \leq \Norm{\alpha_h^{-\frac{1}{2}}}_\infty \Norm{\nabla u} \Norm{\alpha_h - \alpha}.
  \end{equation*}
  Thus, in an a priori setting, one would ideally choose $\alpha_h$ such that
  \begin{equation*}
    \Norm{\alpha_h - \alpha} \sim h^{\min(k,r)}
  \end{equation*}
  to balance the approximation and inconsistency terms in Strang's Lemma. For piecewise-linear basis functions $k=1$ and $r\geq k$ this implies the following a priori relationship between $\alpha_h$ and $h$
  \begin{equation*}
    \Norm{\alpha_h - \alpha} \sim h,
  \end{equation*}
  to ensure consistent asymptotic behaviour.
\end{remark}

\begin{remark}[Target Approximation]
It is relevant to consider how well the discrete
state $u_h$ approximates the desired target $d$. Such estimates have been considered in \cite{langer2024adaptive,langer2022robust} for optimal control problems using energy (meaning $H^{-1}$-norm) regularisation where $\alpha_h |_T \sim h_T^2$ . 
As shown in \cite[Theorem 4.1]{neumuller2021regularization}, for an optimal control problem using an $\leb{2}$ regularisation norm, regularisation error estimates of the
form
\begin{equation*}
  \Norm{u-d} \leq c \, \alpha^\frac{s}{4} \Norm{d}_{\sobh{s}(\Omega)}, \quad s\in
  [0,1],
\end{equation*}
can be derived when extra regularity is imposed on the target $d$. These estimates indicate that if $\alpha$ is chosen too large, the target may never
be well resolved. For the problem under consideration in this work, assuming $d\in \sobhO(\Omega)$, one can build on the results above to show
\begin{equation*}
  \Norm{ u_h - d} \leq \qp{C_1(\alpha) h^{\min(k,r)} + C_2 \alpha^{\frac{s}{4}} }\Norm{d}_{\sobh{s}(\Omega)},
\end{equation*}
although this estimate may be sharpened using appropriate duality arguments.
\end{remark}

\section{Residual Error Estimates}  \label{sec:4}
In this section, we show an a posteriori upper and lower bound on the
error. The section is primarily focussed on the proof of the following
two theorems:

\begin{theorem}[Upper bound]
  \label{the:upper}
  Let $(\fem{u}, \fem{z}) \in X_h$ be the finite element solution to
  \eqref{eq:bhform}. Let the residual and jump estimators be defined
  by
  \begin{align*}
    \begin{split}
      R_z^2 &:= h_T^2 \norm{d + \Delta \fem{z} - \fem{u}}_T^2
      \ \ \qquad 
      J_z^2 := h_E \norm{\jump{\fem{z}}}_E^2,
      \\ 
      R_u^2 &:= h_T^2 \alpha^{-1} \norm{\alpha \Delta \fem{u} + \fem{z}}_T^2 \qquad
      J_u^2 := h_E \alpha^{-1} \norm{\alpha \jump{\fem{u}}}_E^2
    \end{split}
  \end{align*}
  Let the regularisation and discretisation estimators respectively
  be defined by
  \begin{align} 
    \eta_{\alpha, T}^2 &:= \norm{\alpha^{-\frac{1}{2}} (\alpha_h - \alpha) \nabla \fem{u} }_T^2, \label{eq:upestimalpha} \\
    \eta_{h, T}^2 &:= R_z^2 + R_u^2 + \frac{1}{2} \left( \sum_{E \in \mathcal{E}_T} J_u^2 + J_z^2 \right) \label{eq:upestimh}
  \end{align}
  and define the global estimators as
\begin{align*}    
    \eta_\alpha^2 := \sum_{T \in \mathcal{T}} \eta_{\alpha, T}^2 \qquad
    \eta_h^2 :=\sum_{T \in \mathcal{T}} \eta_{h, T}^2.
\end{align*}
  Then, there exists a constant $\widetilde{C}$, independent of $h$,
  $\alpha$, and $\alpha_h$, such that
  \begin{align}
    \|\vec{e}\|_X^2 
    &\leq
    \widetilde{C} 
    \left( 
    \eta_\alpha^2
    +
    \eta_h^2
    \right).
\end{align}  
\end{theorem}

\begin{remark}[A posteriori inconsistency]
  Compare the structure of the a posteriori inconsistency term
  $\eta_{\alpha, T}$ with the a priori version in Lemma
  \ref{lem:strang}. The key idea behind the a posteriori methodology
  we describe is to refine both $\alpha_h$ and $h$ in such a way that
  the inconsistent regularisation is balanced with the approximation
  error.
\end{remark}

\begin{theorem}[Lower bound]
  \label{the:lower}
  Using the notation from Theorem \ref{the:upper}, let $d_h$ be an
  arbitrary function in $V_h$. We assume $u \in \sobh{2}(\Omega)$ and $z \in \sobh{2}(\Omega)$.
  Then, the following local lower error
  bounds hold:
  \begin{align}
    R_z^2 + R_u^2  &\leq   c_1 h_T^2 \norm{d-d_h}_T^2 + c_2 \qp{\alpha \Norm{\nabla e_u}_T^2 + \Norm{\nabla e_z}_T^2} + c_3  h_T^2 \alpha^{-1} \qp{\alpha \norm{e_u}_T^2 + \norm{e_z}_T^2}, \label{eq:Rlowbound}
    \\
    J_u^2 + J_z^2 &\leq c_5 h_E^2 \Norm{d-d_h}_{\widehat{T}}^2 +  c_4 \qp{\alpha \Norm{\nabla e_u}_{\widehat{T}}^2 + \Norm{\nabla e_z}_{\widehat{T}}^2} +  c_5  h_E^2 \alpha^{-1} \qp{\alpha \Norm{e_u}_T^2 + \Norm{e_z}_{\widehat{T}}^2}, 
    \label{eq:Jlowbound} \\
    \eta_{\alpha, T}
  & \leq \qp{\frac{1}{\alpha}-1}\Norm{\alpha^{\frac{1}{2}}\nabla e_u}_T + \alpha^{-\frac{1}{2}} \Norm{\alpha_{h}^{-\frac{1}{2}}(\alpha_h - \alpha) \nabla u}_T.\label{eq:alphalowbound}
\end{align}
where $c_1, \dots, c_5 \in \R$ are strictly positive constants independent of $h, \alpha, \alpha_h$.
\end{theorem}

\begin{remark}[A priori inconsistency vs a posteriori]
Notice that the a posteriori inconsistency term
  $\eta_{\alpha, T}$ is bounded above by its a priori analogue. This is inline with a posteriori studies for inconsistent methods, see for example \cite{vem}. 
\end{remark}

The proof of Theorem \ref{the:upper} and \ref{the:lower} requires the
following technical lemmata:

\begin{lemma}[Scott-Zhang Quasi-Interpolant {\cite[c.f. Theorem 4.1]{scott1990finite}}]  \label[lemma]{lem:clement}
We denote the Scott-Zhang quasi-interpolant by $I_h : \sobhO(\Omega) \to
V_h$. For all $T \in \mathcal{T}$ and $E \in \mathcal{E}$, there exist 
constants $C_{q}, C_{r} > 0$ such that the following error estimates
hold:
\begin{align}
    \norm{v - I_h v}_T &\leq C_q h_T \norm{\nabla v }_{\widetilde{T}}, \\ 
    \norm{v - I_h v}_E &\leq C_r h_E^{ \frac{1}{2}} \norm{\nabla v }_{\widehat{T}},
\end{align}
for all $v \in \sobh1(\widetilde{T})$. The interpolant also satisfies
the following stability property for constant $C_s > 0$ 
\begin{align}
    \|\nabla (I_h v)\|_T \leq C_s \|\nabla v\|_{\widetilde{T}}.
\end{align}
\end{lemma}

\begin{definition}[Lifting Opertator {\cite[page 427]{ern2004theory}}]
  Let $\mathcal{P}_{E,T}: \mathbb{P}_k(E) \to \mathbb{P}_k(T)$ denote the mapping from a facet $E$ to an element $T$. We define the lifting operator $\mathcal{P}_{E}$ from a facet $E$ to the two elements $T^{+}$ and $T^{-}$ sharing this facet to be 
  \begin{align} \label{eq:lifting}
    \forall \phi \in  \mathbb{P}_k(E), \mathcal{P}_{E}: = \begin{cases}
      \mathcal{P}_{E,T^{+}}(\phi) \,\, \text{ on }T^{+} \\
      \mathcal{P}_{E,T^{-}}(\phi) \,\, \text{ on }T^{-}.
    \end{cases} 
  \end{align}
  $\mathcal{P}_{E}$ is supported on $\widehat{T}.$
\end{definition}

\begin{lemma}[Local Cut-Off Functions {\cite[Proposition 1.4]{verfurth2013posteriori}, \cite[Lemma 10.6, Lemma 10.8]{ern2004theory}}] \label{lem:bubble}
Let $\beta_T \in \sobhO(T)$ and $\beta_E \in \sobhO(E)$ denote local cut-off functions (commonly referred to as `bubble' functions) on an element $T$ and facet $E$, respectively. These functions have the following properties:
\begin{align*}
    \text{\emph{supp}}\, \beta_T &= T, & 0 \leq \beta_T &\leq 1, & \norm{\beta_T}_\infty &= 1, \\
    \text{\emph{supp}}\, \beta_E &= \widehat{T}, & 0 \leq \beta_E &\leq 1, & \norm{\beta_E}_\infty &= 1.
\end{align*}
In addition, the following inequalities and inverse estimates hold for all $T\in \mathcal{T}$, $E\in \mathcal{E}$, $\phi \in \mathbb{P}_{m}(T)$ and $\psi \in \mathbb{P}_{n}(T)$
:
\begin{align} \label{eq:bubble1} 
    \norm{\phi}_T &\leq C_{\beta,1} \norm{\beta_T^{\frac{1}{2}} \phi}_T, \\ \label{eq:bubble2} 
    \norm{\nabla \qp{\beta_T \phi}}_T &\leq C_{\beta,2} h_T^{-1} \norm{\phi}_T,  \\ \label{eq:bubble3} 
    \norm{\psi}_E &\leq C_{\beta,3} \norm{\beta_E^{\frac{1}{2}} \psi}_E, \\ \label{eq:bubble4} 
    \norm{\beta_E \mathcal{P}_{E}(\psi)}_{\widehat{T}}&\leq C_{\beta,4} h_E^{\frac{1}{2}} \norm{\psi}_E,
    \\ \label{eq:bubble5} 
    \norm{\nabla \qp{\beta_E \mathcal{P}_{E}(\psi)}}_{\widehat{T}} &\leq C_{\beta,5} h_E^{-\frac{1}{2}} \norm{\psi}_E,  
\end{align}
where the constants $C_{\beta,1}, \dots, C_{\beta,5}$ are mesh-size independent and $m,n\geq 1$ are integers.
\end{lemma}

\subsection*{Proof of Theorem \ref{the:upper}}
We consider the error in the energy norm $\|\cdot\|_X$. Using the
definition of the energy norm \eqref{eq:xnorm} and the variational
problem defined by \eqref{eq:b}-\eqref{eq:bform}, we obtain,
\begin{align*}
  \|\vec e\|_X^2 & = \duality{d}{e_z} - \duality{\alpha  \nabla \fem{u}}{\nabla e_u} +  \duality{\fem{z}}{e_u} - \duality{\nabla \fem{z}}{\nabla e_z} - \duality{\fem{u} }{e_z}.
\end{align*}
Integrating by parts elementwise yields
\begin{equation*}
  \begin{split}
    \|\vec{e}\|_X^2
    =
    \sum_{T\in \mathcal{T}} &\Big (\duality{d + \Delta \fem{z} - \fem{u} }{e_z}_T
    +
    \duality{\alpha\Delta\fem{u} + \fem{z}}{e_u}_T \Big )
    \\  & 
    +
    \sum_{E\in \mathcal{E}}  \Big (\duality{\alpha  \jump{\fem{u}}}{e_u}_E
    +
    \duality{ \jump{\fem{z}}}{e_z}_E \Big ),
  \end{split}
\end{equation*}
where $[\![ \cdot ]\!]$ denotes the jump over a facet $E\in
\mathcal{E}$ as defined in previously. Applying the orthogonality condition from
Theorem \ref{thm:orthogonality} to the above expression we obtain
\begin{equation*}
  \begin{split}
    \|\vec{e}\|_X^2
    =
    &\sum_{T\in \mathcal{T}} \Big (
    \duality{d + \Delta \fem{z} - \fem{u}}{e_z - I_h e_z}_T
    +
    \duality{\alpha\Delta\fem{u} + \fem{z}}{e_u- I_h e_u}_T 
    \\
    &+
    \duality{(\alpha_h - \alpha)\nabla \fem{u}}{\nabla I_h e_u}_T   \Big )
    +
    \sum_{E\in \mathcal{E}} \Big (
    \duality{\alpha  \jump{\fem{u}}}{e_u- I_h e_{u}}_E
    \\
    &+
    \duality{ \jump{\fem{z}}}{e_z- I_h e_{z}}_E \Big ). 
\end{split}
\end{equation*}
Applying the Cauchy-Schwarz inequality results in,
\begin{equation*}
\begin{split}
  \|\vec{e}\|_X^2
  \leq
  \sum_{T\in \mathcal{T}}
  & \Big ( \norm{d + \Delta \fem{z} - \fem{u}}_T \norm{e_z - I_h e_{z}}_T
  +
  \norm{\alpha\Delta\fem{u} + \fem{z}}_T \norm{e_u- I_h e_{u}}_T
  \\ &
  +
  \norm{(\alpha_h - \alpha)\nabla \fem{u} }_T \norm{\nabla I_h e_{u}}_T \Big )
  +
  \sum_{E\in \mathcal{E}} \Big (
  \norm{\alpha \jump{\fem{u}}}_E
  \norm{e_u- I_h e_{u}}_E
  \\&+
  \norm{\jump{\fem{z}}}_E \norm{e_z- I_h e_{z}}_E \Big ). 
\end{split}
\end{equation*}
Now, using the error and stability properties from Lemma \ref{lem:clement} we obtain,
\begin{equation*}
  \begin{split}
     \|\vec{e}\|_X^2 \leq C_I \Bigg ( \sum_{T\in \mathcal{T}} & \Big ( h_T \norm{d + \Delta \fem{z} - \fem{u}}_T \|\nabla e_z \|_{\widetilde{T}} + h_T\norm{\alpha\Delta\fem{u} + \fem{z}}_T \|\nabla e_u \|_{\widetilde{T}} \\ &
     + \norm{(\alpha_h - \alpha)\nabla \fem{u} }_T \norm{\nabla e_{u}}_{\widetilde{T}} \Big )
     + \sum_{E\in \mathcal{E}} \Big ( h_E^{\frac{1}{2}} \norm{\alpha \jump{\fem{u}}}_E \|\nabla e_u \|_{\widehat{T}} \\&+ h_E^{\frac{1}{2}}\norm{\jump{\fem{z}}}_E \|\nabla e_z \|_{\widehat{T}}\Big ) \Bigg ),
\end{split}
\end{equation*}
where $C_I:= \max(C_q, C_r, C_s)$.  Applying the Cauchy-Schwarz
inequality for finite sums twice
yields,
\begin{equation*}
\begin{split}
    \|\vec{e}\|_X^2 &\leq CC_I \Bigg (\sum_{T\in \mathcal{T}} h_T^2 \norm{d + \Delta \fem{z} - \fem{u}}_T^2 + h_T^2 \alpha^{-1} \norm{\alpha\Delta\fem{u} + \fem{z}}_T^2 + \alpha^{-1} \norm{(\alpha_h - \alpha)\nabla \fem{u} }_T^2  
    \\& \,\,\,+ \frac{1}{2} \bigg (\sum_{E\in \mathcal{T}} h_E \alpha^{-1}
    \norm{\alpha \jump{\fem{u}}}_E^2 + h_E  \norm{\jump{\fem{z}}}_E^2 \bigg )
    \Bigg )^{\frac{1}{2}} \Bigg ( \|\nabla e_z \|^2 + \alpha  \|\nabla e_u \|^2
    \Bigg )^{\frac{1}{2}},
\end{split}
\end{equation*}
where $C$ is a constant that takes into account the repeated counting
of elements on a patch. Therefore, with $\widetilde{C} := C\max(C_q,
C_r, C_s)$, we have the upper bound,
\begin{align*}
    \|\vec{e}\|_X \leq  \widetilde{C} \Bigg ( \sum_{T\in \mathcal{T}} \eta_{h, T}^2 + \eta_{\alpha, T}^2   \Bigg)^{\frac{1}{2}},
\end{align*}
as required. \qed

\subsection*{Proof of Theorem \ref{the:lower}}

We split this proof into three components: the interior component of
$\eta_{h,T}^2$, the facet component of $\eta_{h,T}^2$, and the
regularisation estimator $\eta_{\alpha, T}^2$.

\paragraph{1. The Interior Component of $\eta_{h,T}^2$: $R_u^2 + R_z^2$}

We introduce an arbitrary function $d_h \in V_h$. 
Expanding $R_z$ using the triangle inequality,
\begin{align*}
    \norm{d + \Delta \fem{z} - \fem{u}}_T\leq \norm{d_h + \Delta \fem{z} - \fem{u}}_T + \norm{d-d_h}_T. 
\end{align*}
Let $r_{zh} := d_h + \Delta \fem{z} - \fem{u}$. Introducing the bubble function $\beta_T$ for an element $T$ and using \eqref{eq:bubble1}, we obtain
\begin{align*}
    \norm{r_{zh}}_T^2 
    &\leq C_{\beta,1}^2 \int_T (d_h + \Delta \fem{z} - \fem{u}) \beta_T r_{zh}.
\end{align*}
Recalling that $\Delta z - u +d =0$ and using integration by parts yields,
\begin{align*}
    \norm{r_{zh}}_T^2 
    &\leq C_{\beta,1}^2 \qp{ \int_T \qp{(d_h - d) + (u -\fem{u})} \beta_T r_{zh} + \nabla (z- \fem{z}) \cdot \nabla (\beta_T r_{zh}) }.
\end{align*}
Applying the Cauchy-Schwarz inequality, the inverse estimate given in \eqref{eq:bubble2} and bounding by $\norm{\beta_T}_\infty= 1$ results in,
\begin{align*}
    \norm{r_{zh}}_T 
    &\leq C_{\beta,1}^2 \bigg (\norm{d - d_h}_T + \norm{u -\fem{u}}_T + C_{\beta,2}h_T^{-1}\norm{\nabla (z- \fem{z})}_T\bigg ).
\end{align*}
Therefore,
\begin{align} \label{eq:zp1}
    \norm{d + \Delta \fem{z} - \fem{u}}_T\leq C_{\beta,1}^2 \bigg (\norm{u -\fem{u}}_T + C_{\beta,2}h_T^{-1}\norm{\nabla (z- \fem{z})}_T\bigg ) + (1+ C_{\beta,1}^2)\norm{d-d_h}_T. 
\end{align}
Squaring this expression, making use of the Cauchy-Schwarz inequality for finite sums, and multiplying by $h_T^2$
\begin{align} \label{eq:zp2}
    R_z^2 
    &\leq 3C_{\beta,1}^4 \bigg (h_T^2\norm{u -\fem{u}}_T^2 + C_{\beta,2}^2\norm{\nabla (z- \fem{z})}_T^2\bigg ) + 3(1+C_{\beta,1}^2)^2 h_T^2 \norm{d-d_h}_T^2. 
\end{align}
Using the same arguments as above for $R_u^2$ yields the inequalities,
\begin{align} \label{eq:up1}
    \norm{\alpha \Delta \fem{u}+ \fem{z}}_T 
    &\leq C_{\beta,1}^2 \Bigg (\norm{z -\fem{z}}_T + C_{\beta,2}h_T^{-1}\alpha \norm{\nabla (u- \fem{u})}_T\bigg ),
\end{align}
and 
\begin{align} \label{eq:up2}
    R_u^2
    &\leq 2C_{\beta,1}^4 \Bigg ( \alpha^{-1} h_T^2 \norm{z -\fem{z}}_T^2 + C_{\beta,2}^2 \alpha \norm{\nabla (u- \fem{u})}_T^2\bigg ).
\end{align}
Combining \eqref{eq:zp2} and \eqref{eq:up2}, we obtain the following bound,
\begin{equation*}
\begin{split}
        R_z^2 + R_u^2  \leq & \, 3\qp{1+ C_{\beta,1}^2}^2 h_T^2 \norm{d-d_h}_T^2  \\&+3C_{\beta,1}^4 \qp{ C_{\beta,2}^2 \qp{\alpha \Norm{\nabla e_u}_T^2 + \Norm{\nabla e_z}_T^2} + \alpha^{-1} h_T^2 \qp{\alpha \norm{e_u}_T^2 + \norm{e_z}_T^2}}.
\end{split}
\end{equation*}
\paragraph{2. The Facet Component of $\eta_{h,T}^2$: $J_u^2 + J_z^2$}
We now bound $\norm{\jump{\fem{z}}}_E$. We introduce a bubble function on a facet $E$. Using \eqref{eq:bubble3}, 
\begin{equation*}
    \norm{\jump{\fem{z}}}_E^2 \leq C_{\beta,3}^2 \norm{\beta_E^{\frac{1}{2}}\jump{\fem{z}}}_E^2.
\end{equation*} 
We note that as $z \in \sobh{2}(\Omega)$ this implies that $\jump{z} |_E = 0$. Using this fact and introducing the lifting operator defined in \eqref{eq:lifting}, 
\begin{align*}
    \norm{\jump{\fem{z}}}_E^2 \leq C_{\beta,3}^2 \int_E \jump{(\fem{z}-z)} \beta_E \mathcal{P}_{E} \qp{\jump{\fem{z}}}.
\end{align*}
As $\beta_E \mathcal{P}_{E}\qp{\jump{\fem{z}}}$ is only supported on $\widehat{T}$ and is continuous across $E$, we obtain use integration by parts to obtain the following
\begin{align*}
    C_{\beta,3}^{-2}\norm{\jump{\fem{z}}}_E^2 \leq \int_{\widehat{T}} \nabla \qp{z-z_h} \cdot \nabla \qp{\beta_E \mathcal{P}_{E}\qp{\jump{\fem{z}}}} + \int_{\widehat{T}} \beta_E \mathcal{P}_{E}\qp{\jump{\fem{z}}} \Delta (z-z_h)
\end{align*}
Applying the Cauchy-Schwartz inequality yields
\begin{equation*}
    \begin{split}
    C_{\beta,3}^{-2}\norm{\jump{\fem{z}}}_E^2 \leq &\Norm{\nabla \qp{z-z_h}}_{\widehat{T}} \Norm{\nabla \qp{\beta_E \mathcal{P}_{E}\qp{\jump{\fem{z}}}}}_{\widehat{T}} \\ &+ \Norm{\beta_E \mathcal{P}_{E}\qp{\jump{\fem{z}}} }_{\widehat{T}} \Norm{\Delta (z-z_h)}_{\widehat{T}}.
    \end{split}
\end{equation*}
Applying \eqref{eq:bubble4} and \eqref{eq:bubble5} yields 
\begin{equation*}
    \begin{split}
    C_{\beta,3}^{-2}\norm{\jump{\fem{z}}}_E^2 \leq C_{\beta,5} h_{E}^{-\frac{1}{2}}\Norm{\nabla \qp{z-z_h}}_{\widehat{T}} \Norm{\jump{\fem{z}}}_{E}  + C_{\beta, 4} h_{E}^{\frac{1}{2}} \Norm{\Delta (z-z_h)}_{\widehat{T}} \Norm{\jump{\fem{z}}}_{E}.
    \end{split}
\end{equation*}
Recalling that $\Delta z - u +d =0$ and using the triangle inequality produces
\begin{equation*}
    \begin{split}
    C_{\beta,3}^{-2}\norm{\jump{\fem{z}}}_E \leq C_{\beta,5} h_{E}^{-\frac{1}{2}}\Norm{\nabla \qp{z-z_h}}_{\widehat{T}}   + C_{\beta, 4} h_{E}^{\frac{1}{2}} \qp{\Norm{\Delta\fem{z}- u_h + d}_{\widehat{T}} + \Norm{u-u_h}_{\widehat{T}} }.
    \end{split}
\end{equation*}
Using the estimate derived in \eqref{eq:zp1}, we find that 
\begin{equation}
    \begin{split}
    \norm{\jump{\fem{z}}}_E \leq &C_{\beta,3}^{2}\qp{C_{\beta,5} + C_{\beta, 4}C_{\beta,1}^2 C_{\beta,2}}h_{E}^{- \frac{1}{2}} \Norm{\nabla e_z}_{\widehat{T}} \\ &+ C_{\beta,3}^{2} C_{\beta,4}\qp{1+C_{\beta,1}^2} h_{E}^{\frac{1}{2}} \Norm{e_u}_{\hat{T}} +  C_{\beta,3}^{2} C_{\beta,4}\qp{1+C_{\beta,1}^2} h_{E}^{\frac{1}{2}} \Norm{d-d_h}_{\hat{T}}.
    \end{split}
    \label{eq:exp1}
\end{equation}
Using the same arguments for $\norm{\alpha\jump{\fem{u}}}_E$ yields
\begin{equation}
    \begin{split}
    \norm{\alpha \jump{\fem{u}}}_E \leq &\, C_{\beta,3}^{2}\qp{C_{\beta,5} + C_{\beta, 4}C_{\beta,1}^2 C_{\beta,2}}h_{E}^{- \frac{1}{2}} \alpha \Norm{\nabla e_u}_{\widehat{T}} + C_{\beta,3}^{2} C_{\beta,4}\qp{1+C_{\beta,1}^2} h_{E}^{\frac{1}{2}} \Norm{e_z}_{\widehat{T}}.
    \end{split}
    \label{eq:exp2}
\end{equation}
Combining \eqref{eq:exp1} and \eqref{eq:exp2}, we obtain the bound
\begin{equation*}
\begin{split}
    J_u^2 + J_z^2 \leq &\, 3  C_{\beta,3}^{4}\qp{C_{\beta,5} + C_{\beta, 4}C_{\beta,1}^2 C_{\beta,2}}^2 \qp{\alpha \Norm{\nabla e_u}_{\widehat{T}}^2 + \Norm{\nabla e_z}_{\widehat{T}}^2} \\
    &+ 3  C_{\beta,3}^{4} C_{\beta,4}^2\qp{1+ C_{\beta,1}^2}^2 h_E^2 \alpha^{-1} \qp{\alpha \Norm{e_u}_T^2 + \Norm{e_z}_{\widehat{T}}^2} \\
    &+ 3  C_{\beta,3}^{4} C_{\beta,4}^2\qp{1+ C_{\beta,1}^2}^2 h_E^2 \Norm{d-d_h}_{\widehat{T}}^2.
\end{split}
\end{equation*}

\paragraph{3. The regularisation estimator $\eta_{\alpha,T}^2$} 
We expand $\eta_{\alpha,T}$ using the triangle inequality, adding and subtracting $\alpha^{-\frac{1}{2}}(\alpha_h - \alpha) \nabla u$ in the norm,
\begin{equation*}
  \norm{\alpha^{-\frac{1}{2}} (\alpha_h - \alpha)\nabla \fem{u} }_T
  \leq \Norm{\qp{\frac{\alpha_h}{\alpha}-1} \alpha^{\frac{1}{2}}\nabla e_u}_T + \Norm{\qp{\frac{\alpha_h}{\alpha}}^{\frac{1}{2}} \alpha_{h}^{-\frac{1}{2}}(\alpha_h - \alpha) \nabla u}_T.
\end{equation*}
Recall that $\alpha \leq \alpha_h \leq 1$.
Therefore, we have the bound
\begin{equation*}
  \norm{\alpha^{-\frac{1}{2}} (\alpha_h - \alpha)\nabla \fem{u} }_T
  \leq \qp{\frac{1}{\alpha}-1}\Norm{\alpha^{\frac{1}{2}}\nabla e_u}_T + \alpha^{-\frac{1}{2}} \Norm{\alpha_{h}^{-\frac{1}{2}}(\alpha_h - \alpha) \nabla u}_T.
\end{equation*}
We note that the second term is the natural relative data oscillation
of the regularisation parameter.

Finally, defining the following group of constants,
  \begin{align*}
    c_1 &= 3\qp{1+ C_{\beta,1}^2}^2, & c_2 &= 3C_{\beta,1}^4 C_{\beta,2}^2, & c_3 &= 3C_{\beta,1}^4, \\ c_4 &=3  C_{\beta,3}^{4}\qp{C_{\beta,5} + C_{\beta, 4}C_{\beta,1}^2 C_{\beta,2}}^2  ,  & c_5 &= C_{\beta,3}^{4} C_{\beta,4}^2 c_1,
\end{align*}
we arrive at the bounds stated in \eqref{eq:Rlowbound}-\eqref{eq:alphalowbound}
as required. \qed

\section{Numerical Results}  \label{sec:results}
We draw inspiration from adaptive finite element methods
\cite{bonito2024adaptive} to develop our numerical scheme. These
methods typically employ local a posteriori error estimators to
determine which elements require refinement. The process follows an
adaptive loop
\begin{equation*}
    \textsc{solve } \rightarrow \textsc{ estimate } \rightarrow
    \textsc{ mark } \rightarrow \textsc{ refine}.
\end{equation*}
First, the variational problem is solved using a finite element method
on a coarse mesh. Second, error estimators are evaluated on
each element. Third, the elements with the largest error estimators
are marked for refinement, based on a marking strategy. Common marking
strategies include D\"orfler marking, maximum strategy, and
equidistribution strategy \cite{nochetto2009theory}. Finally, the
marked elements are refined using a bisection method. This loop is
repeated until the desired resolution is achieved.

We apply a similar scheme for adaptively selecting the regularisation
parameter. Starting with a relatively large initial regularisation
parameter, $\alpha_h^0 \,\,\sim \mathcal{O}(1)\,$, we compute the estimator $\eta_{\alpha,T}$
\eqref{eq:upestimalpha} for each $T \in \mathcal{T}$. The
regularisation parameter $\alpha_h$ is then reduced for each element,
relative to the size of the estimator. 
In places where the estimator is large, $\alpha_h$ is reduced the most.
This loop continues until a desired error tolerance is achieved.

Algorithm \ref{alg:1} outlines the proposed
adaptive regularisation numerical scheme in conjunction with adaptive
mesh refinement. 
This algorithm balances the regularisation and discretisation errors. 
Two error tolerances $\texttt{tol}_\alpha \text{ and } \texttt{tol}_h$ are provided for the regularisation error and discretisation error respectively. The adaptive method is executed until both estimators are below their error tolerances. Once the error tolerance for one of the quantities is reached, we stop the refinement process for this quantity and continue refining the other until its convergence criteria is satisfied. 

We have chosen to use maximum strategy to mark the elements for adaptive mesh refinement. We exclude any elements where the inconsistency
error $\eta_\alpha$ is larger than the discretisation error $\eta_h$. 
As will be demonstrated in \autoref{sec:conv} Figure \ref{fig:1c}, the inconsistency error dominates over the discretisation error. Removing these elements from mesh refinement avoids refining the mesh 
unnecessarily in places where the error could have been decreased by reducing $\alpha_h$.

As this is an optimal control problem, we are naturally interested in finding the control variable.
The finite element approximation of the control variable, $f_h$, can be reconstructed in a post-processing step after the adaptive scheme has concluded. 
We choose to find $f_h$ in the space of piecewise-linear polynomials, however, other choices of finite element spaces could be used so long as $f_h \in \leb{2}(\Omega)$.
We introduce the global finite element space
\begin{equation*}
    W_h := \{ w_h \in C^0(\Omega) : w_h|_T \in \mathcal{R}_1(T) \; \forall \, T \in \mathcal{T} \}.
\end{equation*}
We discretise \eqref{eq:KKT3} and seek $f_h \in W_h$ such that
\begin{equation*}
   \duality{\alpha_h f_h - z_h}{\tau_h} = 0, \threespace \forall \tau_h \in W_h.
\end{equation*}
It is clear from this equation that the control is given by the $\alpha_h$-weighted $\leb{2}$-projection of $z_h / \alpha_h$ into the space $W_h$, or equally, an interpolation operator could be chosen.

\begin{algorithm}[h!]
\caption{Adaptive Finite Element Method with Regularisation Refinement}
\label{alg:1}
\textbf{Input: }$\alpha_h^0$, initial coarse mesh, regularisation reduction factor $\rho \, \in [0,1]$, mesh refinement factor $\theta \in [0,1]$, $\texttt{tol}_\alpha >0$, $\texttt{tol}_h >0$.
\begin{algorithmic}[1]
    \While{$\eta_\alpha \geq \texttt{tol}_\alpha$ \textbf{and} $\eta_h \geq \texttt{tol}_h$}
        \State \textbf{Solve:} Solve the discrete problem \eqref{eq:bhform} on the current mesh.
        \State \textbf{Estimate:} Compute the error estimators $\eta_h$ and $\eta_{\alpha}$.
        \State \textbf{Regularisation Refinement:}
        \If{$\eta_\alpha \geq \texttt{tol}_\alpha$}
            \For{each element $T \in \mathcal{T}$}
            \State Compute the relative size of the estimator:
            \begin{equation*}
                r_T = \frac{\eta_{\alpha, T} -  \eta_\alpha^{\min}}{\eta_\alpha^{\max} - \eta_\alpha^{\min}}
            \end{equation*}
        \State where $\eta_{\alpha}^{\max} = \max_{T\in \mathcal{T}} \, \eta_{\alpha,T}$ and $\eta_{\alpha}^{\min} = \min_{T\in \mathcal{T}} \, \eta_{\alpha,T}$.
        \EndFor
        \ElsIf{$\eta_\alpha^{\max} =  \eta_\alpha^{\min}$  \textbf{ or } $\eta_\alpha < \texttt{tol}_\alpha$}
        \State $r_T = 0$ for all $T \in \mathcal{T}$.
        \EndIf

        \State Reduce $\alpha_h$:
        \begin{equation*}
            \alpha_{h}^{\text{new}} |_T = \alpha \, + (\alpha_{h} |_T - \alpha)(1 + r_T (\rho-1))
        \end{equation*}
        
        \State \textbf{Mesh Refinement:} 
        \If{$\eta_h \geq \texttt{tol}_h$}
        \State \textbf{Mark:} Use maximum strategy to mark elements for $h$-refinement. 
        \State The marked elements are in the set
        \begin{equation*}
            M := \{ T \in \mathcal{T} \, |  \, \eta_{h,T} \geq \theta \eta_h^{\max} \text{ and } \eta_{h,T} \geq \eta_{\alpha,T}\} 
        \end{equation*}
        \State where $\eta_{h}^{\max} = \max_{T\in \mathcal{T}} \, \eta_{h,T}$. 
        \State \textbf{Refine} the marked elements $M$ using the newest vertex bisection method.
        \State Project $\alpha_h^{\text{new}}$ onto the new mesh.
        \EndIf
        \State Redefine $\alpha_h = \alpha_h^{\text{new}}$

    \EndWhile
    \State Reconstruct the control variable using the weighted $\leb{2}$-projection:
    \begin{equation*}
        f_h = \arg \min_{\tau_h \in W_h} \Norm{\alpha_h^{\frac{1}{2}}\qp{\tau_h - \frac{z_h}{\alpha_h}}}
    \end{equation*}
\end{algorithmic}
\end{algorithm}

We now construct two examples with manufactured solutions to test our
numerical algorithm.  For simplicity of presentation, we consider both
cases on a one-dimensional domain $\Omega = (0,1).$

\example[Smooth Target]{\label{gaussian} Let $K_1=500$ and define $y:=
  -K_1(x-0.5)^2$ and $K_2:= \qp{1 + 12\alpha K_1^2}^{-1}$.  We define
  the target by the following function,
  \begin{equation*}
    d(x) := K_2\qp{1 + \alpha 4K_1^2 \qp{4y^2 + 12y + 3}}\exp(y).
  \end{equation*}
  The solution to \eqref{eq:bform} with this target are given by
  \begin{align*}
    u(x) &= K_2\exp(y), &
    z(x) &=  \alpha 2K_1 \qp{2y - 1} K_2\exp(y)
  \end{align*}
  This example is chosen as the support of the solution is localised
  to the centre of the domain. Plots of the target and solutions are
  shown in \autoref{fig:gaussian}.}

\begin{figure}[h!]
  \centering
  \begin{subfigure}{0.32\textwidth}
    \includegraphics[width=\linewidth]{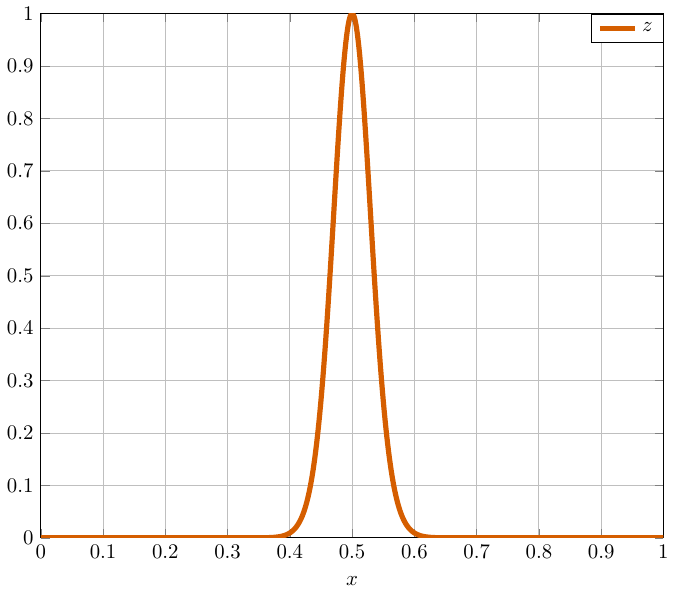}
    \caption{The target $d(x)$.}
    \label{fig:map}
  \end{subfigure}\hfil
  \begin{subfigure}{0.32\textwidth}
    \includegraphics[width=\linewidth]{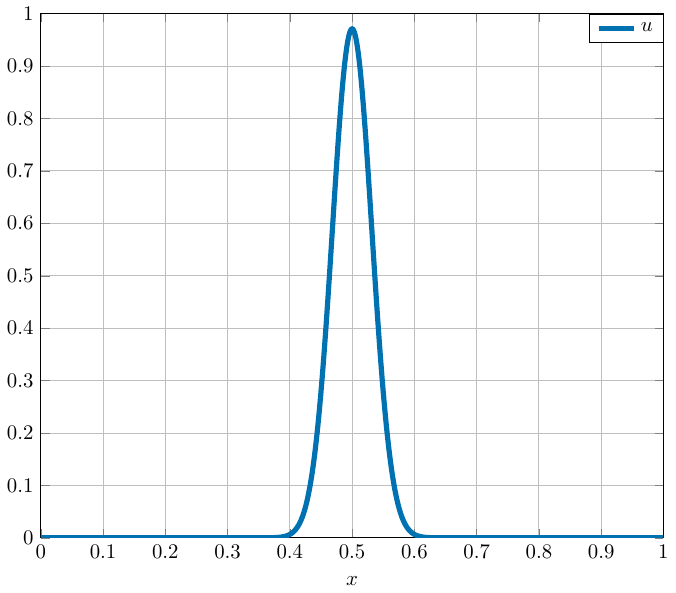}
    \caption{The primal solution $u(x)$.}
    \label{fig:pcd}
  \end{subfigure}\hfil
    \begin{subfigure}{0.32\textwidth}
    \includegraphics[width=\linewidth]{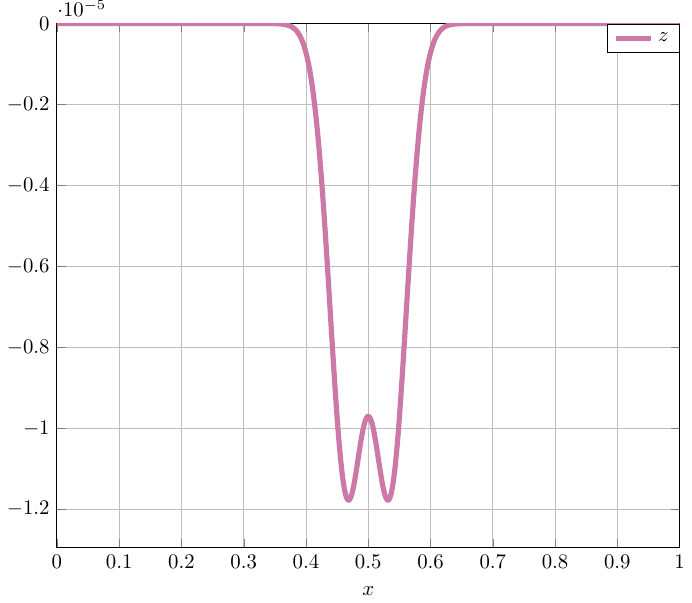}
    \caption{The adjoint solution $z(x)$.}
    \label{fig:pcd}
  \end{subfigure}\hfil
  \caption{Example \ref{gaussian}. Plots of the target, state, and
    adjoint variables with $\alpha = 10^{-8}$.}
  \label{fig:gaussian}
\end{figure}

\example[Boundary Layers \cite{MakridakisPimPryer:2024}]{
\label{bl}
Let $d(x) = 1$ and define $\omega := \qp{4\alpha}^{-1/4}$, $y=
\omega x$, and 
\begin{align*}
    K_3 &= \frac{\sinh(\omega)}{\cosh(\omega) + \cos(\omega)}, & K_4 &= \frac{\sin(\omega)}{\cosh(\omega) + \cos(\omega)}.
\end{align*}
We can write a closed form solution to this problem as
\begin{align*}
    u(x) &= 1 - \cosh(y) \cos(y) + K_3\sinh(y) \cos(y) - K_4 \cosh(y) \sin(y), 
    \\
    z(x) &=-\alpha^{\frac{1}{2}} \qp{\sinh(y) \sin(y) - K_3\cosh(y) \sin(y) - K_4\sinh(y) \cos(y)}.
\end{align*}
This problem is chosen as it induces boundary layers for small
$\alpha$. The solution is plotted in \autoref{fig:bl}.  }

\begin{figure}[h!]
  \centering
  \begin{subfigure}{0.35\textwidth}
    \includegraphics[width=\linewidth]{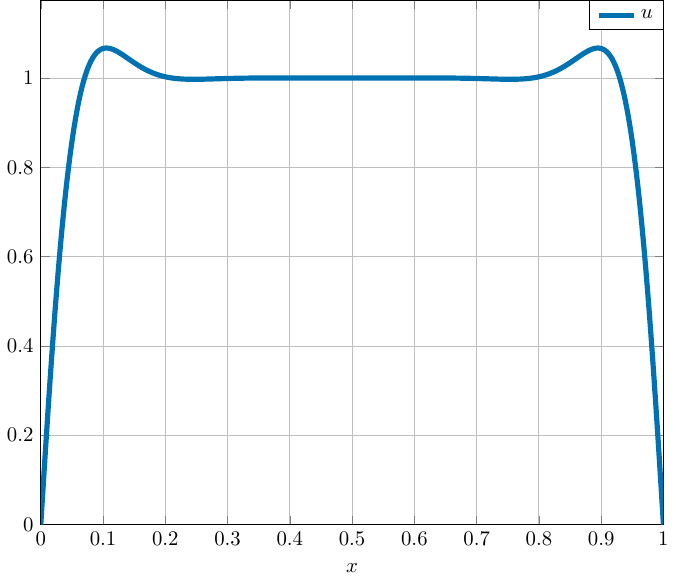}
    \caption{$u(x)$}
    \label{fig:map}
  \end{subfigure}\hfil
  \begin{subfigure}{0.35\textwidth}
    \includegraphics[width=\linewidth]{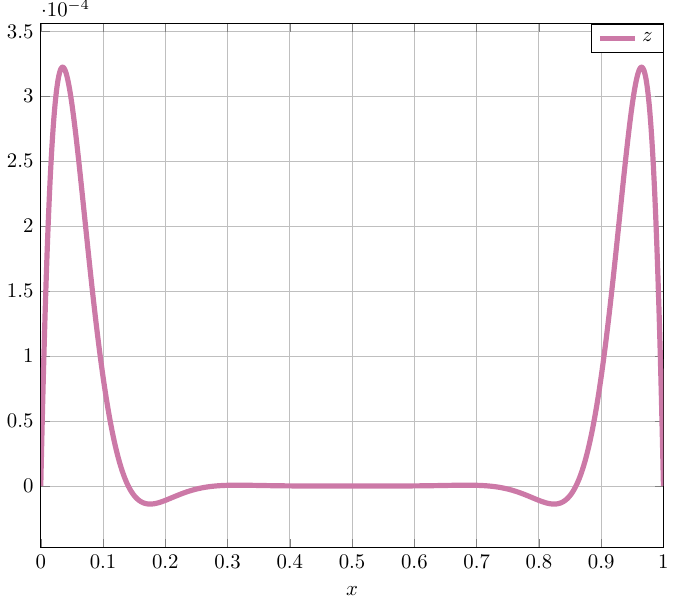}
    \caption{$z(x)$}
    \label{fig:pcd}
  \end{subfigure}\hfil
  \caption{Example \ref{bl}. Plots of state, and adjoint variables, $\alpha = 10^{-6}$.}
  \label{fig:bl}
\end{figure}

\subsection{Convergence Analysis} \label{sec:conv}
All numerical experiments were implemented using FEniCS \cite{AlnaesEtal2015, LoggEtal2012}. We begin by analysing the behaviour of the estimators introduced in
\autoref{the:upper}. Three tests are conducted using examples
\ref{gaussian} and \ref{bl}, with the results shown in Figure
\ref{fig:conv_2}. The first test
investigates the convergence of the mesh component of the estimator
$\eta_h$. The adaptive regularisation parameter is chosen such that
$\alpha_h = \alpha$, ensuring the regularisation estimator $\eta_\alpha$ is
zero. The results are shown in Figures \ref{fig:1a} and \ref{fig:2a}. Both the error and estimator are evaluated as the mesh is
uniformly refined. For both cases, the error and estimator converge
linearly, in agreement with established theory.

The second test computes the effectivity index $\eta_h  /
\|\vec{e}\|_X$ as $\alpha$ decreases. A fixed uniform mesh was used
with $\text{dim}V_h = 10^5$ and the adaptive regularisation parameter
was again chosen as $\alpha_h = \alpha$. The results are displayed in
Figures \ref{fig:1b} and \ref{fig:2b}. In both examples, the
effectivity remains constant (approximately 4.9) for all values of
$\alpha$ considered. We conclude that the estimator $\eta_h$ is robust
with respect to $\alpha$.

Finally, we examine the behaviour of the inconsistency term $\eta_\alpha$. The mesh was uniform with $\text{dim}V_h = 10^5$. The error, global regularisation and discretisation estimators, and the full estimate were evaluated as $\alpha_h$ was decreased uniformly to $\alpha$. The results can be seen in Figures 
\ref{fig:1c} and \ref{fig:2c}. 
All four quantities decrease as $\alpha_h$ is decreased. When $\alpha_h \neq \alpha$, the regularisation inconsistency estimate dominates over the discretisation estimate. Once $\alpha_h = \alpha$, the regularisation estimator vanishes, and the error is bounded by the discretisation estimate only, as is the case in Figures \ref{fig:1a} and \ref{fig:2a}. 

\begin{figure}[h!]
  \centering
  \begin{subfigure}[t]{0.32\textwidth}
    \includegraphics[width=\linewidth]{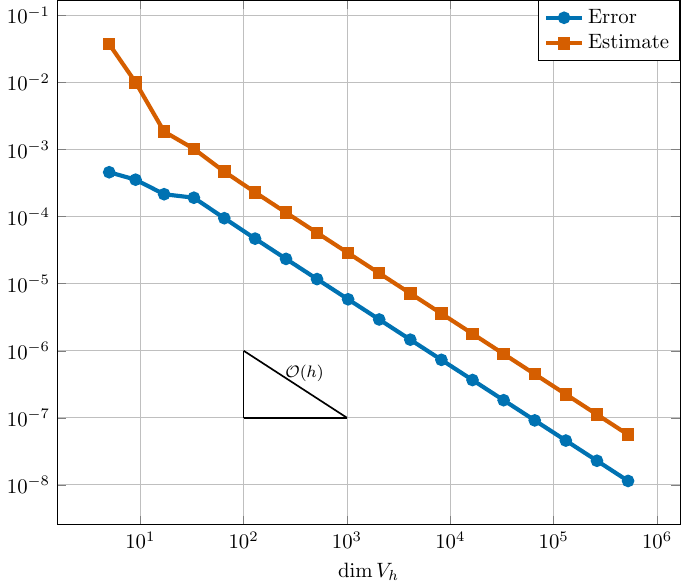}
    \caption{Example \ref{gaussian}: The error and estimator for fixed $\alpha_h = \alpha = 10^{-8}$.}
    \label{fig:1a}
  \end{subfigure}\hfil
  \begin{subfigure}[t]{0.32\textwidth}
    \includegraphics[width=\linewidth]{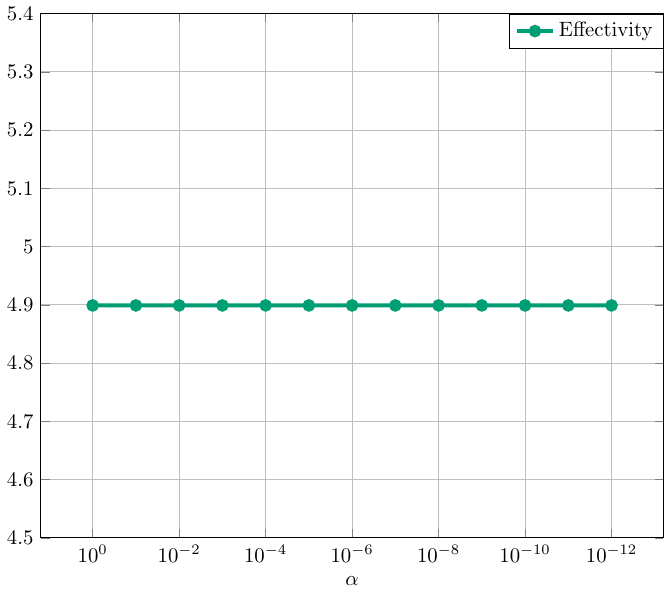}
    \caption{Example \ref{gaussian}: Effectivity index vs $\alpha$. Fixed mesh $\text{dim }V_h = 10^5$ and $\alpha_h = \alpha$. }
    \label{fig:1b}
  \end{subfigure}\hfil
    \begin{subfigure}[t]{0.32\textwidth}
    \includegraphics[width=\linewidth]{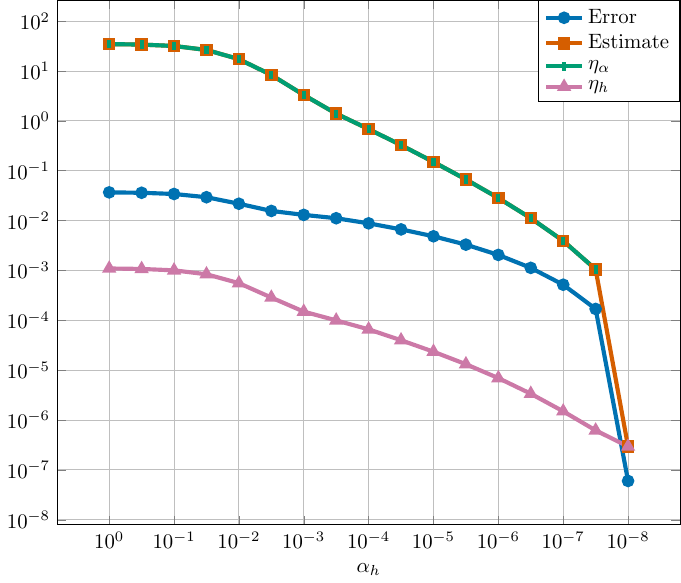}
    \caption{Example \ref{gaussian}: Error and estimator for fixed mesh, $\text{dim }V_h = 10^5, \alpha = 10^{-8}$.}
    \label{fig:1c}
  \end{subfigure}\hfil
  \begin{subfigure}[t]{0.32\textwidth}
    \includegraphics[width=\linewidth]{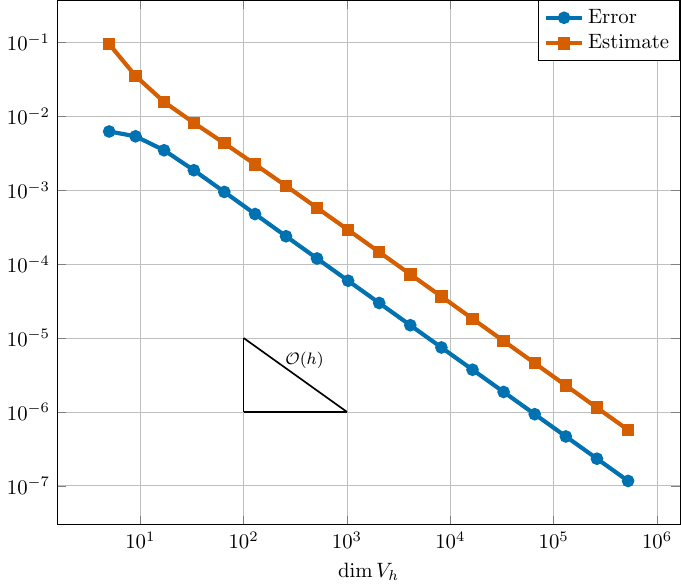}
    \caption{Example \ref{bl}: The error and estimator for fixed $\alpha_h = \alpha = 10^{-6}$.}
    \label{fig:2a}
  \end{subfigure}\hfil
  \begin{subfigure}[t]{0.32\textwidth}
    \includegraphics[width=\linewidth]{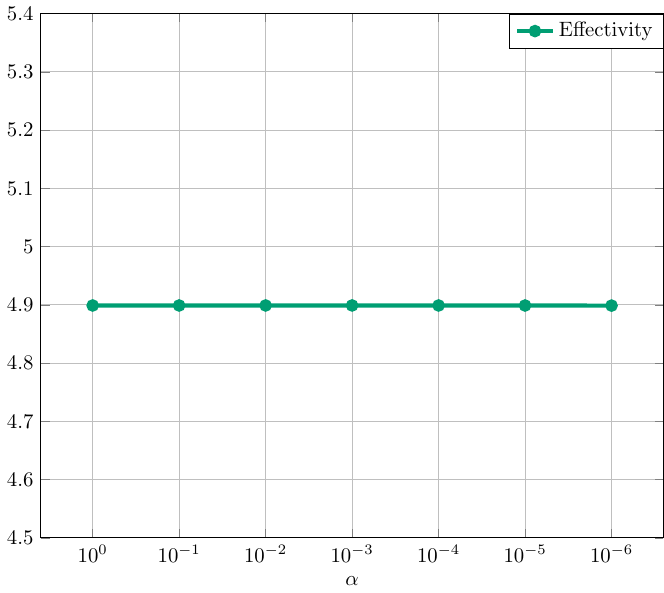}
    \caption{Example \ref{bl}: Effectivity index vs $\alpha$. Fixed mesh $\text{dim }V_h = 10^5$ and $\alpha_h = \alpha$. }
    \label{fig:2b}
  \end{subfigure}\hfil
    \begin{subfigure}[t]{0.32\textwidth}
    \includegraphics[width=\linewidth]{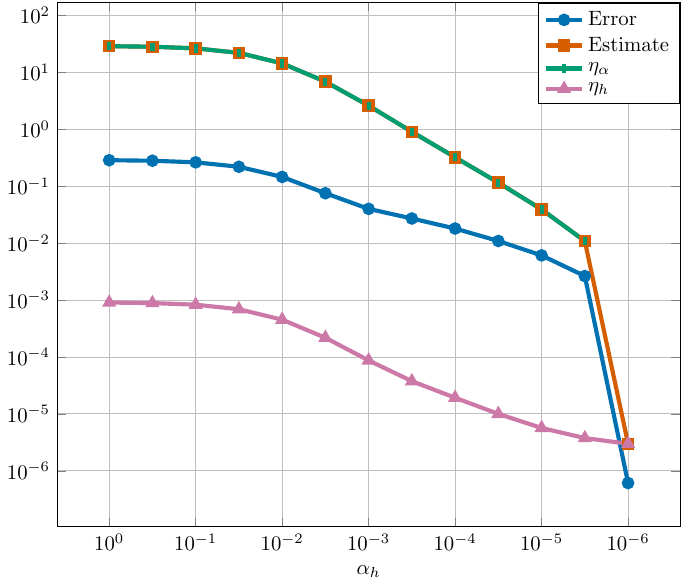}
    \caption{Example \ref{bl}: Error and estimator for fixed mesh, $\text{dim }V_h = 10^5, \alpha = 10^{-6}$.}
    \label{fig:2c}
  \end{subfigure}\hfil
  \caption{Convergence plots for \autoref{sec:conv} for examples \ref{gaussian} and \ref{bl}. The left column showcases the convergence of the mesh estimator $\eta_T$ as the mesh is uniformly refined with no adaptive regularisation. The results showcase that this estimator is robust in $h$. The middle column plots the effectivity index $\eta_h/\|\vec{e}\|_X$ for a fixed uniform mesh for different choices of $\alpha$. The effectivity is constant and therefore $\eta_h$ is robust in $\alpha$. The right column showcases the behaviour of the error and estimators on a fixed uniform mesh for the adaptive regularisation scheme as $\alpha_h$ was decreased uniformly to $\alpha$.}
  \label{fig:conv_2}
\end{figure}

\subsection{Regularisation Refinement Adaptive Scheme}
\label{sec:regadpt}

We examine the behaviour of Algorithm \ref{alg:1} without mesh
refinement, focusing solely on the adaptive regularisation component.
The mesh was uniform with $\text{dim}V_h = 10^4$. For both examples,
the regularisation reduction factor was set to $\rho = 0.5$, and an
initial value of $\alpha_h^0 = 1$ was used. A convergence criterion of
$\eta_\alpha < \texttt{tol}_{\alpha}$ was applied. The results for example
\ref{gaussian} are presented in Figures \ref{fig:gaussian_alphaonly}
and \ref{fig:gaussian_alphaonly_stills}. The regularisation parameter
was set to $\alpha = 10^{-8}$, with an error tolerance of
$\texttt{tol} = 10^{-8}$, and convergence was achieved after 201 iterations.

Plots of the adaptive regularisation parameter $\alpha_h$, the
regularisation estimator $\eta_{\alpha,T}$, and the approximate state
variable $u_h$ over the domain are shown in Figure
\ref{fig:gaussian_alphaonly_stills} for several iterations. Accompanying videos can also be found in the supplementary material.
From the
plots, it is clear that $\alpha_h$ tends to $\alpha$ primarily in regions with
large gradients, which is consistent with the structure of the
inconsistency term $\eta_{\alpha, T}$. In regions where the solution
is zero, $\alpha_h$ is chosen to be larger than $\alpha$, with
$\max \alpha_h \sim \mathcal{O}(10^{-4})$, almost four orders of magnitude
greater than $\alpha$. The corresponding convergence results are shown
in Figure \ref{fig:gaussian_alphaonly}. These plots showcase the behaviour of the error norm, estimator components, effectivity index, and the maximum, minimum, and mean average values of $\alpha_h$ per iteration.

\begin{figure}[h!]
  \centering
  \begin{subfigure}{0.35\textwidth}
    \includegraphics[width=\linewidth]{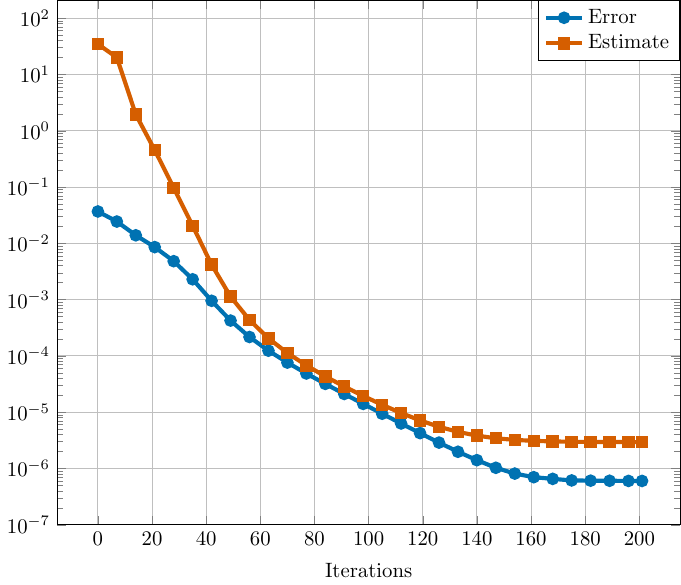}
    \caption{Error and Estimator}
    \label{fig:gaussian_alpha_error}
  \end{subfigure}\hfil
  \begin{subfigure}{0.35\textwidth}
    \includegraphics[width=\linewidth]{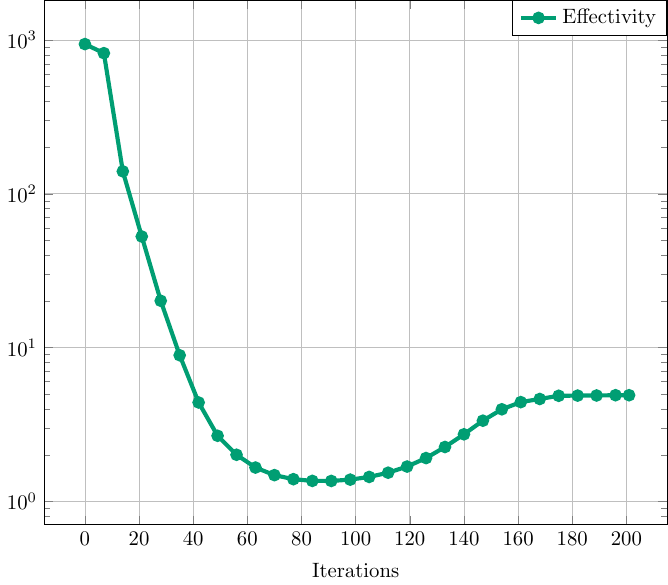}
    \caption{Effectivity}
    \label{fig:gaussian_alpha_ei}
  \end{subfigure}
    \begin{subfigure}{0.35\textwidth}
    \includegraphics[width=\linewidth]{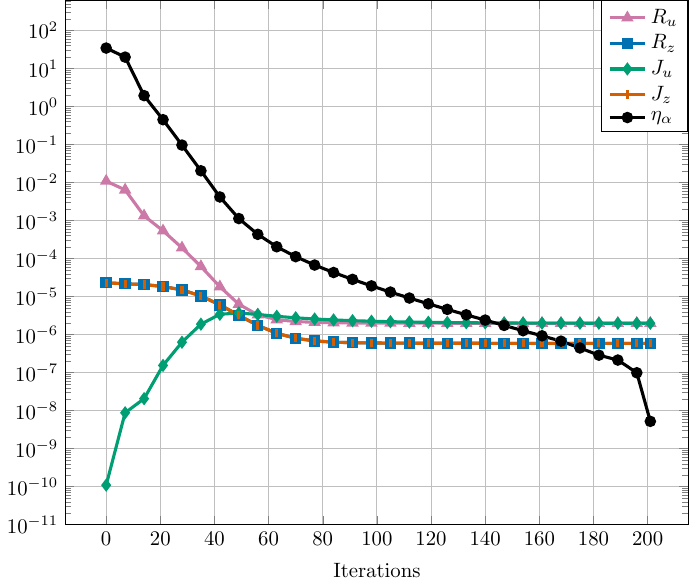}
    \caption{Estimator Components}
    \label{fig:gaussian_alpha_individual}
  \end{subfigure}\hfil
  \begin{subfigure}{0.35\textwidth}
    \includegraphics[width=\linewidth]{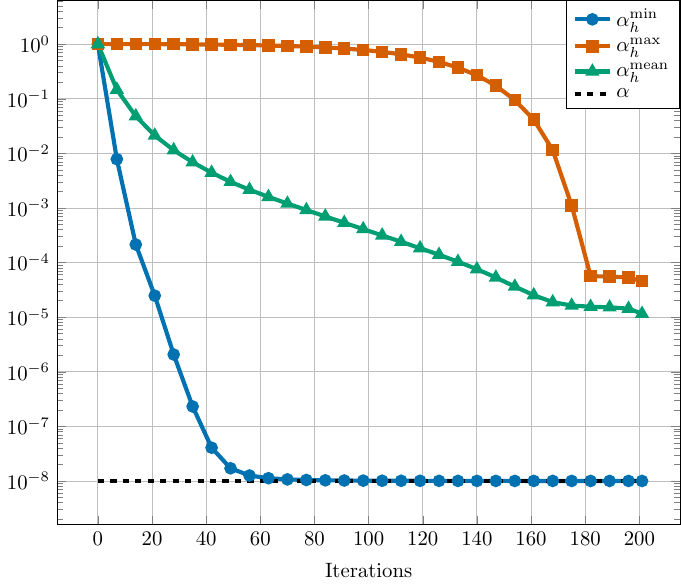}
    \caption{Parameter Ranges}
    \label{fig:gaussian_alpha_params}
  \end{subfigure}
  \caption{Example \ref{gaussian} (smooth target), \autoref{sec:regadpt}. Convergence results for the regularisation refinement only scheme. $\text{dim}V_h = 10^4$, $\rho = 0.5$ and $\alpha = 10^{-8}$, $\texttt{tol}_{\alpha} = 10^{-8}$}.
  \label{fig:gaussian_alphaonly} 
\end{figure}

\begin{figure}[h!]
  \centering
   \begin{subfigure}{\textwidth}
    \includegraphics[width=\linewidth]{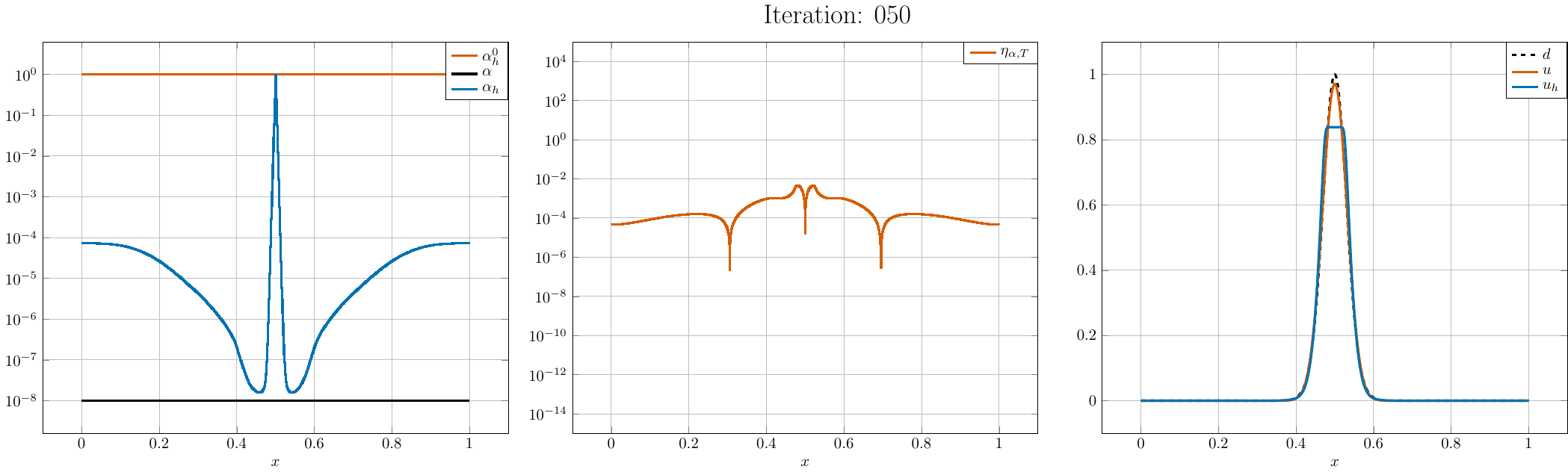}
    \label{fig:map}
  \end{subfigure}
  \begin{subfigure}{\textwidth}
    \includegraphics[width=\linewidth]{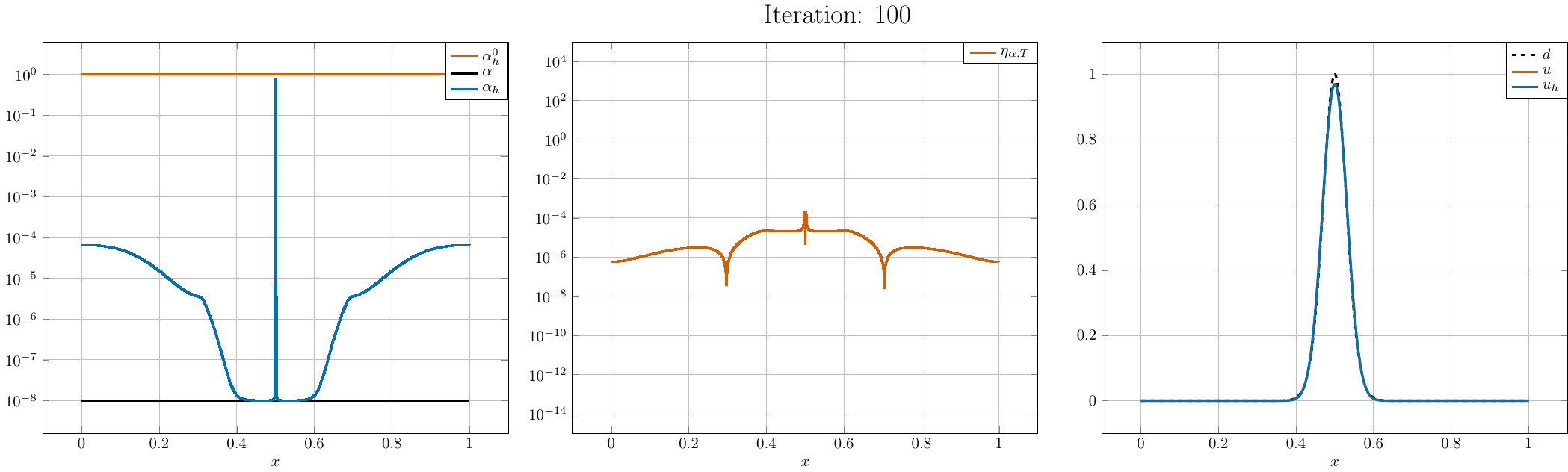}
    \label{}
  \end{subfigure}
    \begin{subfigure}{\textwidth}
    \includegraphics[width=\linewidth]{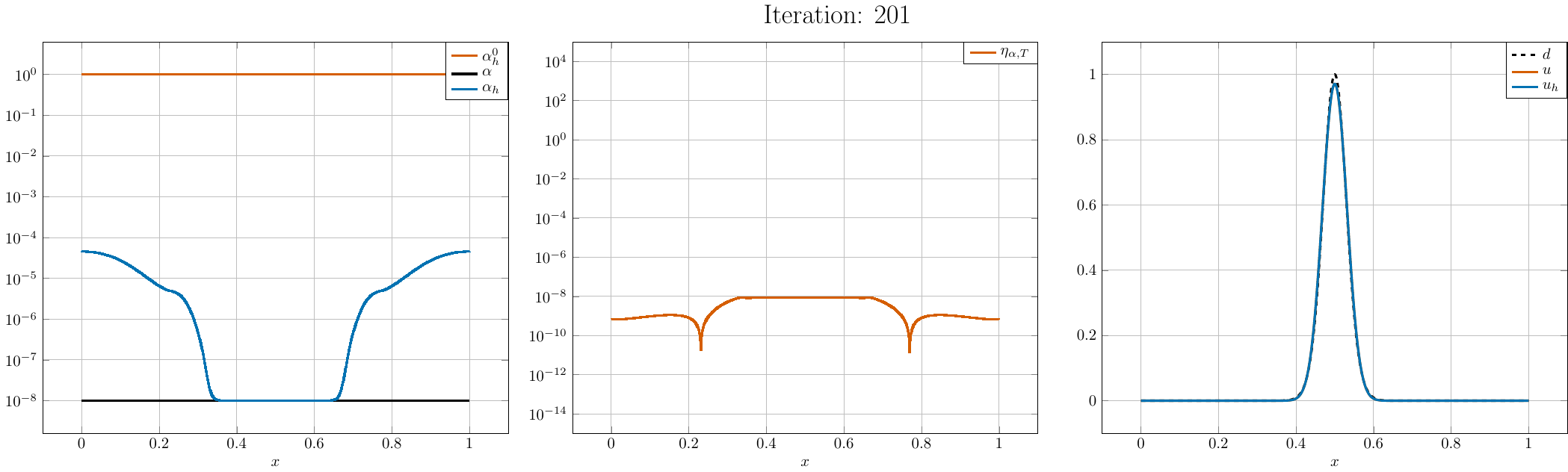}
    \label{fig:}
  \end{subfigure}
  \caption{Example \ref{gaussian} (smooth target): \autoref{sec:regadpt}, regularisation refinement only. Plots of $\alpha_h$ (left),
  $\eta_\alpha$ (middle) and the approximate solution $u_h$ compared to $u$ and $d$ (right) over the domain at three different iterations. $\text{dim}V_h = 10^4$, $\rho = 0.5$ and $\alpha = 10^{-8}$, $\texttt{tol}_\alpha = 10^{-8}$. Videos of these stills for every iteration can be found in the supplementary material.}
  \label{fig:gaussian_alphaonly_stills}
\end{figure}

\begin{figure}[h!]
  \centering
  \begin{subfigure}{0.35\textwidth}
    \includegraphics[width=\linewidth]{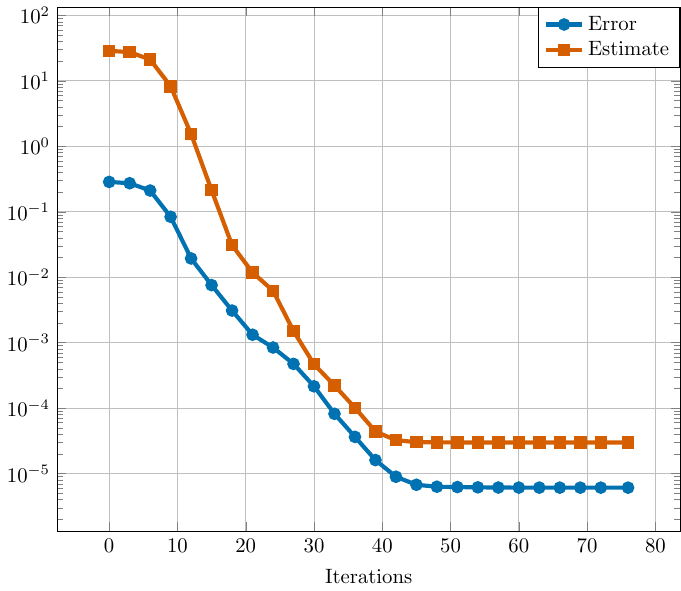}
    \caption{Error and Estimator}
    \label{fig:}
  \end{subfigure}\hfil
  \begin{subfigure}{0.35\textwidth}
    \includegraphics[width=\linewidth]{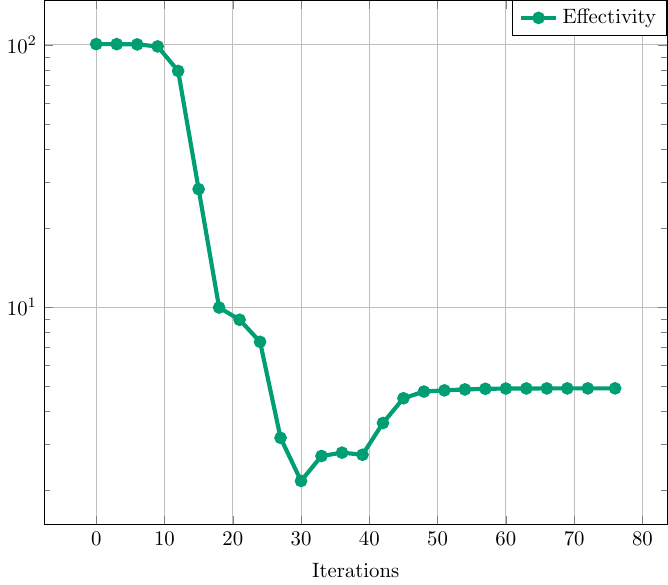}
    \caption{Effectivity}
    \label{fig:}
  \end{subfigure}
    \begin{subfigure}{0.35\textwidth}
    \includegraphics[width=\linewidth]{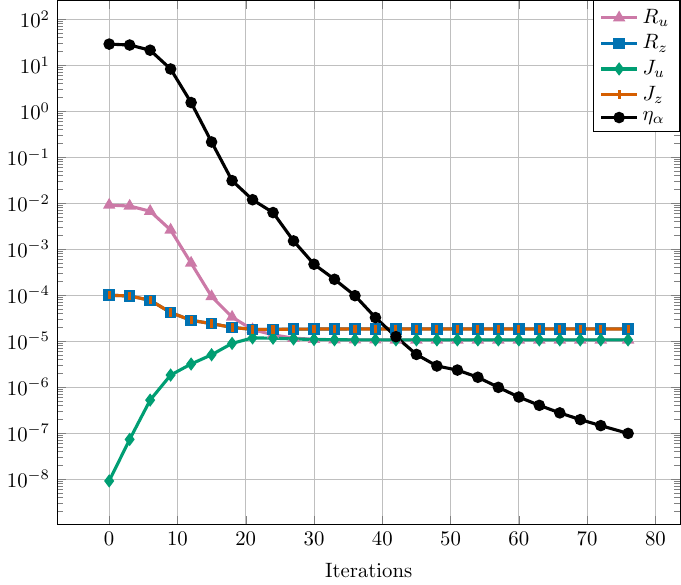}
    \caption{Estimator Components}
    \label{fig:}
  \end{subfigure}\hfil
  \begin{subfigure}{0.35\textwidth}
    \includegraphics[width=\linewidth]{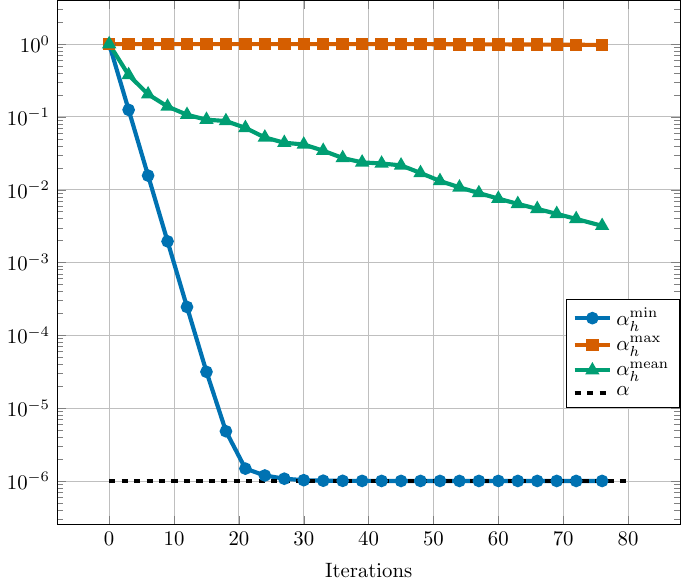}
    \caption{Parameter Ranges}
    \label{fig:}
  \end{subfigure}
  \caption{Example \ref{bl} (boundary layer), \autoref{sec:regadpt}. Convergence results for the regularisation refinement only scheme. $\text{dim}V_h = 10^4$, $\rho = 0.5$ and $\alpha = 10^{-6}$, $\texttt{tol}_\alpha = 10^{-7}$}.
  \label{fig:bl_alphaonly}
\end{figure}

\begin{figure}[h!]
  \centering
  \begin{subfigure}{\textwidth}
    \includegraphics[width=\linewidth]{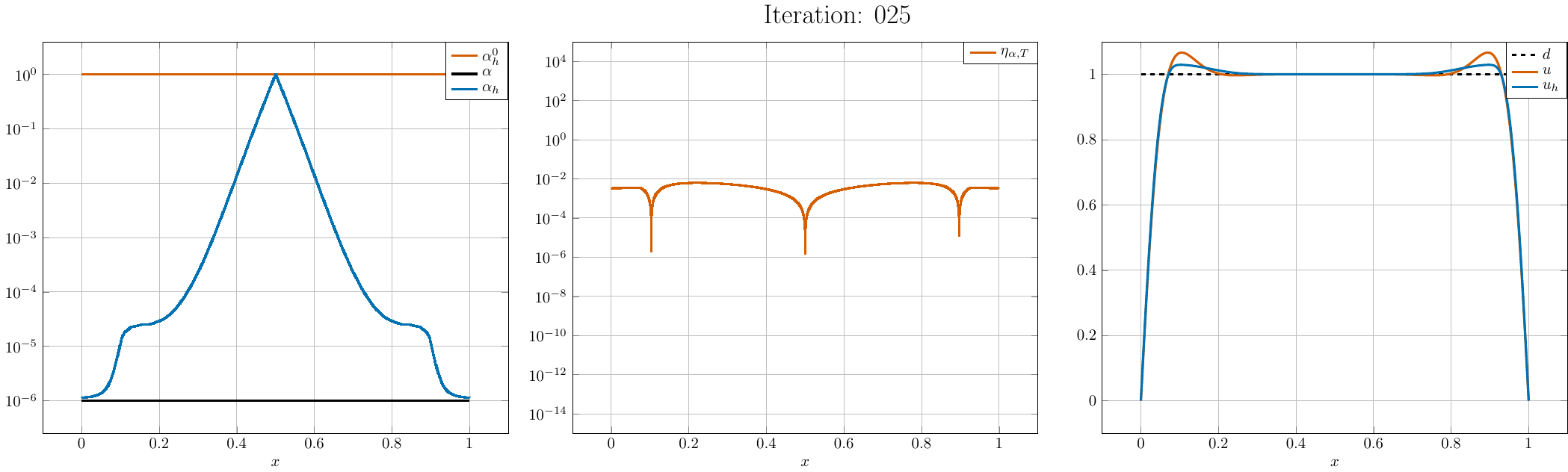}
    \label{fig:}
  \end{subfigure}
  \begin{subfigure}{\textwidth}
    \includegraphics[width=\linewidth]{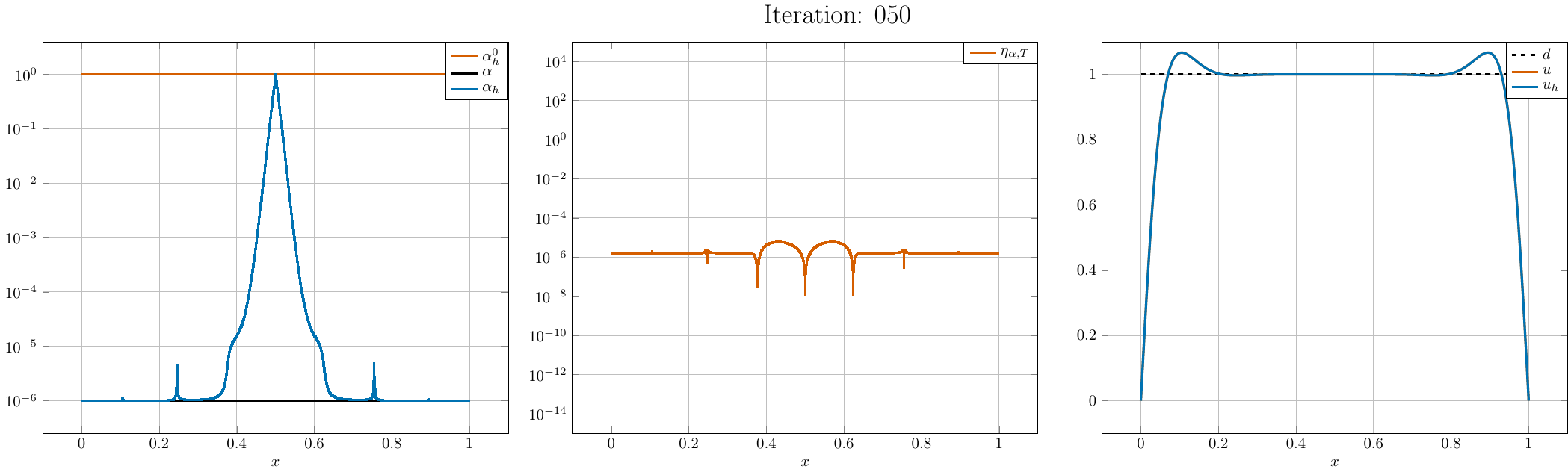}
    \label{}
  \end{subfigure}
    \begin{subfigure}{\textwidth}
    \includegraphics[width=\linewidth]{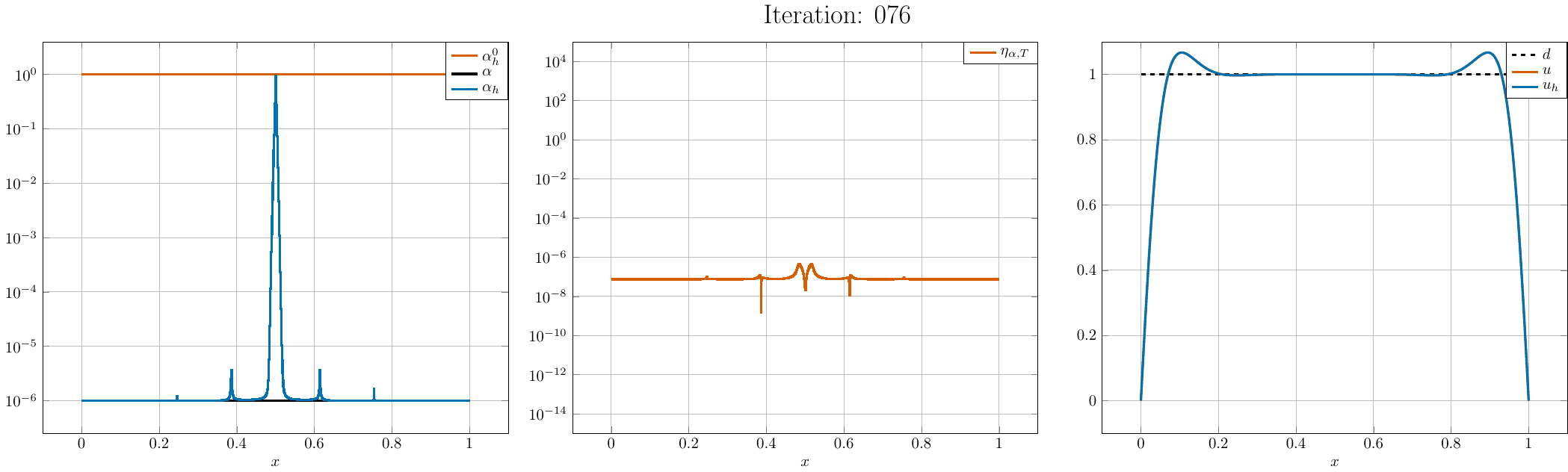}
    \label{fig:}
  \end{subfigure}
  \caption{Example \ref{bl} (boundary layer), \autoref{sec:regadpt}: regularisation refinement only. Plots of $\alpha_h$ (left), $\eta_\alpha$ (middle) and the approximate solution $u_h$ compared to $u$ and $d$ (right) over the domain at four different iterations. $\text{dim}V_h = 10^4$, $\rho = 0.5$ and $\alpha = 10^{-6}$, $\texttt{tol}_\alpha = 10^{-7}$. Videos of these stills for every iteration can be found in the supplementary material.}
  \label{fig:bl_alphaonly_stills}
\end{figure}

Figures \ref{fig:bl_alphaonly} and \ref{fig:bl_alphaonly_stills}
present the results for example \ref{bl}, with the regularization
parameter set to $\alpha = 10^{-6}$. An error tolerance of
$\texttt{tol}_{\alpha} = 10^{-7}$ was used, and convergence was reached after 76 iterations. This example exhibits boundary layer behaviour, with
large gradients near the boundaries.
From the solution plots in Figure \ref{fig:bl_alphaonly_stills}, we
observe that $\alpha_h$ tends toward $\alpha$ in the vicinity of the
boundary, as expected. Away from the boundary region, where the solution
$u$ is constant and the gradient is small, $\alpha_h$ can be chosen to
be larger, with $\max \alpha_h = 1$. Figure
\ref{fig:bl_alphaonly} shows the corresponding convergence results.
It is worth noting that for both examples, the effectivity index stabilizes at
$\qp{\eta_{h}^2 + \eta_{\alpha}^2}^{\frac{1}{2}} / \|\vec{e}\|_X
\approx 4.9$ once convergence is achieved, which is the effectivity index when
$\alpha_h = \alpha$, as identified in the previous section.

\subsection{Mesh and Regularisation Refinement Adaptive Scheme}
\label{sec:fulladaptive}
\begin{figure}[h!]
  \centering
  \begin{subfigure}{0.35\textwidth}
    \includegraphics[width=\linewidth]{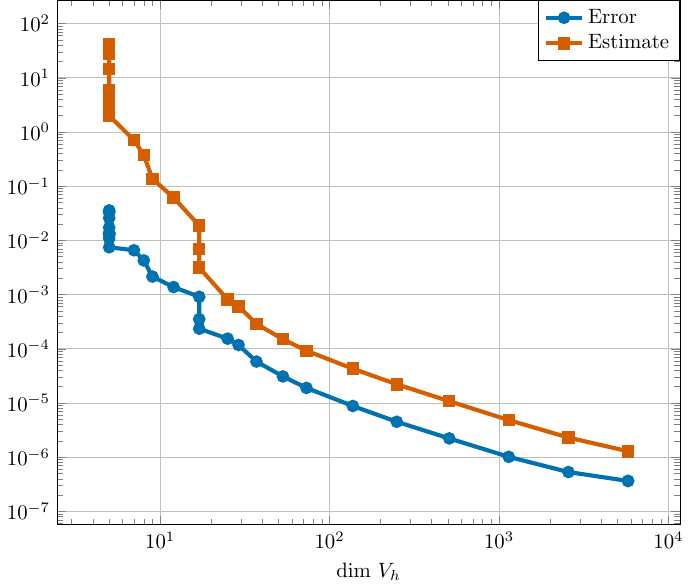}
    \caption{Error and Estimator}
    \label{fig:map}
  \end{subfigure}\hfil
  \begin{subfigure}{0.35\textwidth}
    \includegraphics[width=\linewidth]{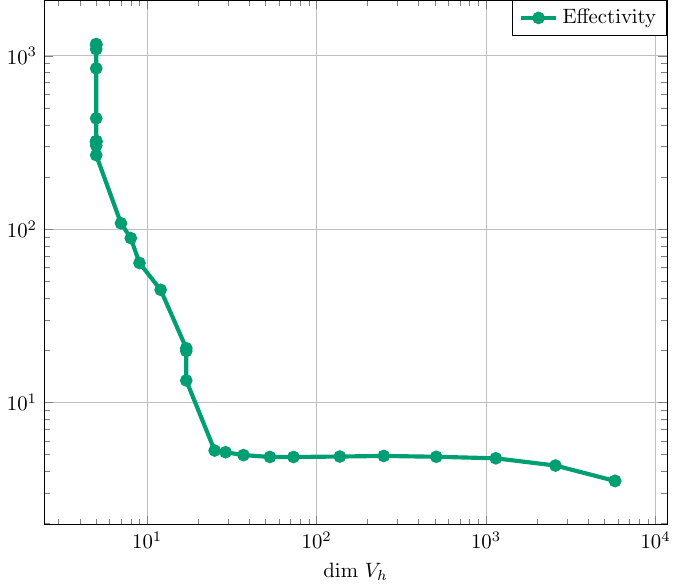}
    \caption{Effectivity}
    \label{fig:pcd}
  \end{subfigure}
    \begin{subfigure}{0.35\textwidth}
    \includegraphics[width=\linewidth]{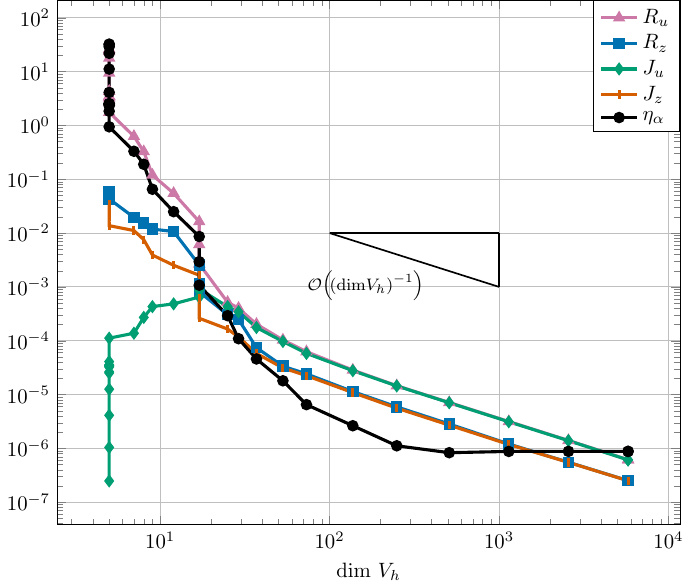}
    \caption{Estimator Components}
    \label{fig:gfindividual}
  \end{subfigure}\hfil
  \begin{subfigure}{0.35\textwidth}
    \includegraphics[width=\linewidth]{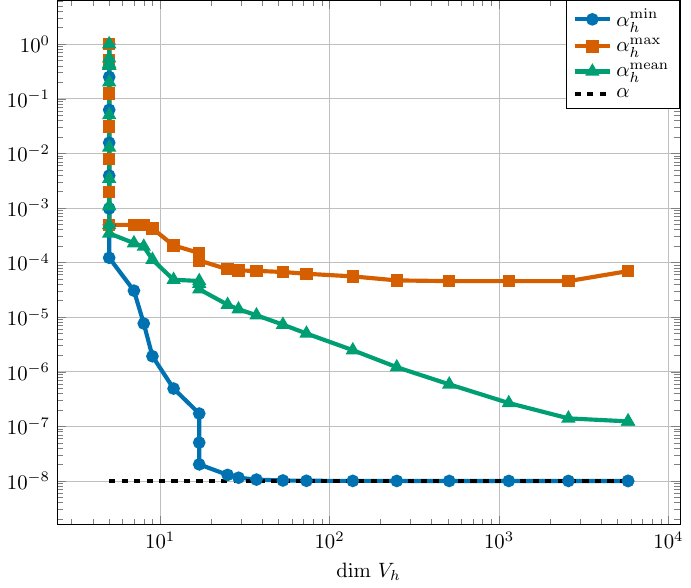}
    \caption{Parameter Ranges}
    \label{fig:pcd}
  \end{subfigure}
  \caption{Example \ref{gaussian} (smooth target), \autoref{sec:fulladaptive}: convergence results for the fully adaptive scheme. $\rho = 0.5$, $\theta = 0.5$ and $\alpha = 10^{-8}$.}
  \label{fig:gaussian_full}
\end{figure}

\begin{figure}[h!]
  \centering
  \begin{subfigure}{\textwidth}
    \includegraphics[width=\linewidth]{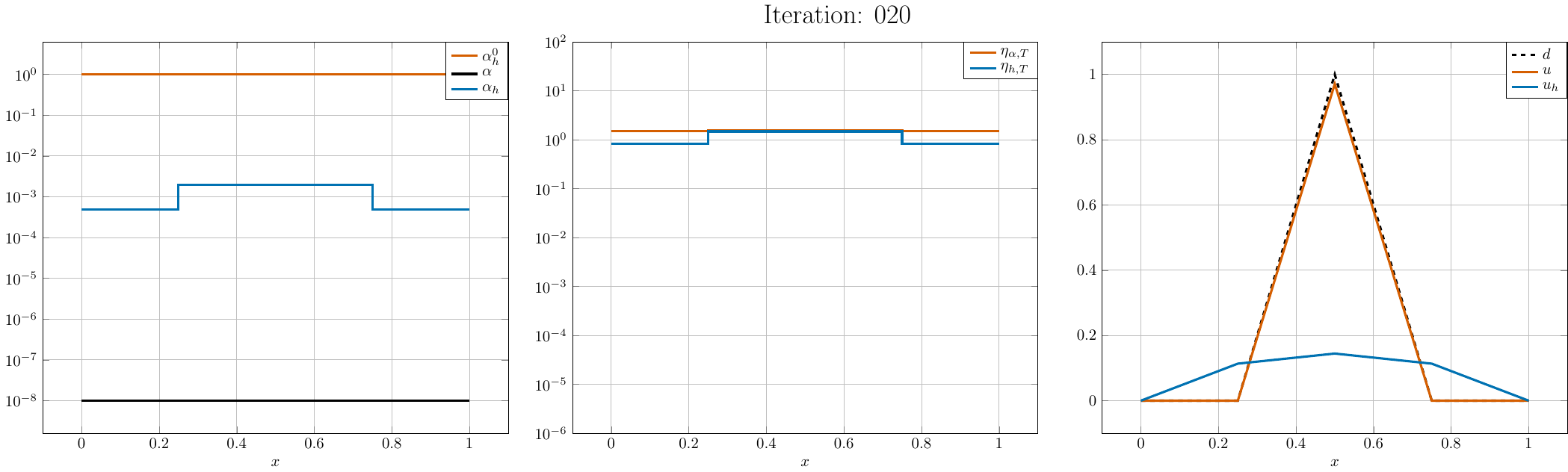}
    \label{fig:map}
  \end{subfigure}
  \begin{subfigure}{\textwidth}
    \includegraphics[width=\linewidth]{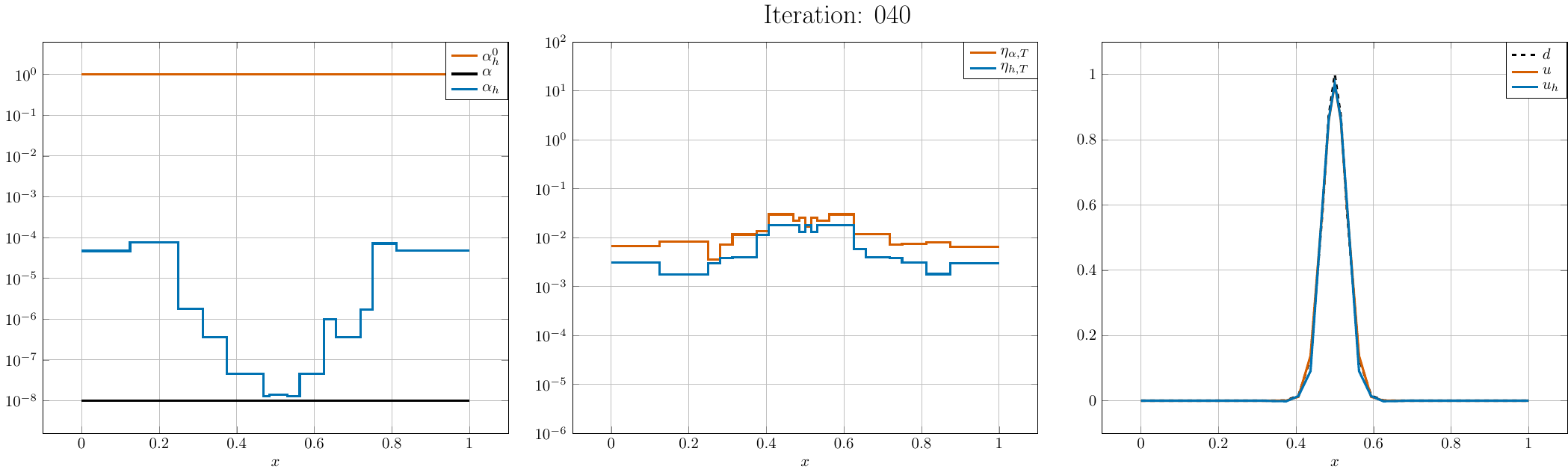}
    \label{}
  \end{subfigure}
    \begin{subfigure}{\textwidth}
    \includegraphics[width=\linewidth]{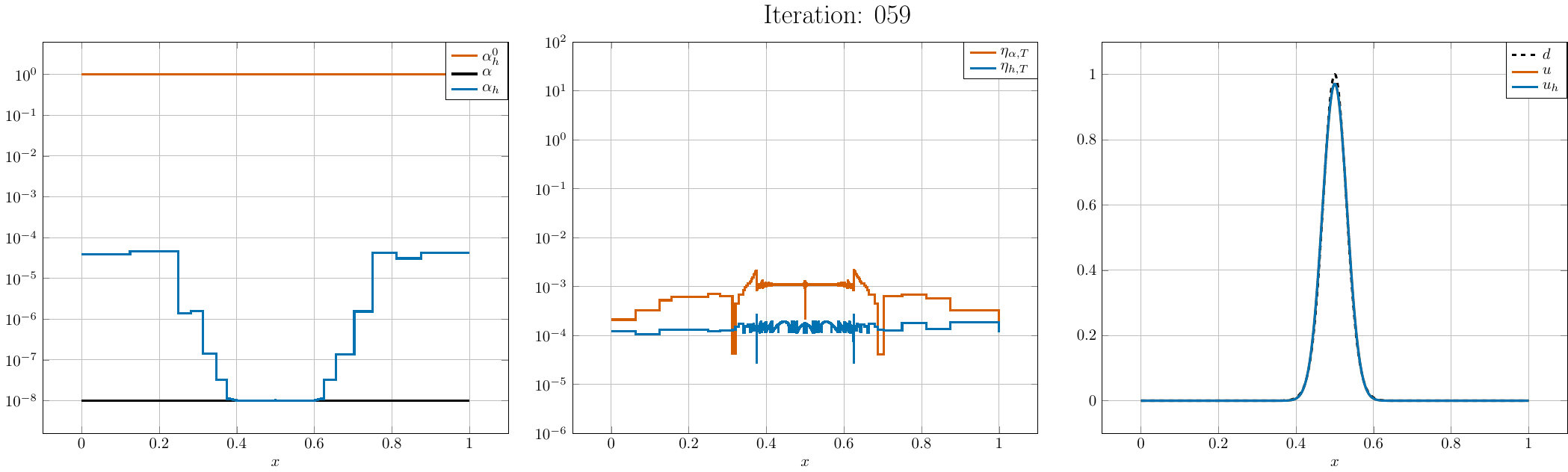}
    \label{fig:}
  \end{subfigure}
  \caption{Example \ref{gaussian} (smooth target), \autoref{sec:fulladaptive}, fully adaptive scheme. Plots of $\alpha_h$ (left), $\eta_\alpha, \eta_h$ (middle) and the approximate solution $u_h$ compared to $u$ and $d$ (right) over the domain at four different iterations. $\rho = 0.5$, $\theta = 0.5$ and $\alpha = 10^{-8}$.}
  \label{fig:gaussian_full_stills}
\end{figure}

\begin{figure}[h!]
  \centering
  \begin{subfigure}{0.35\textwidth}
    \includegraphics[width=\linewidth]{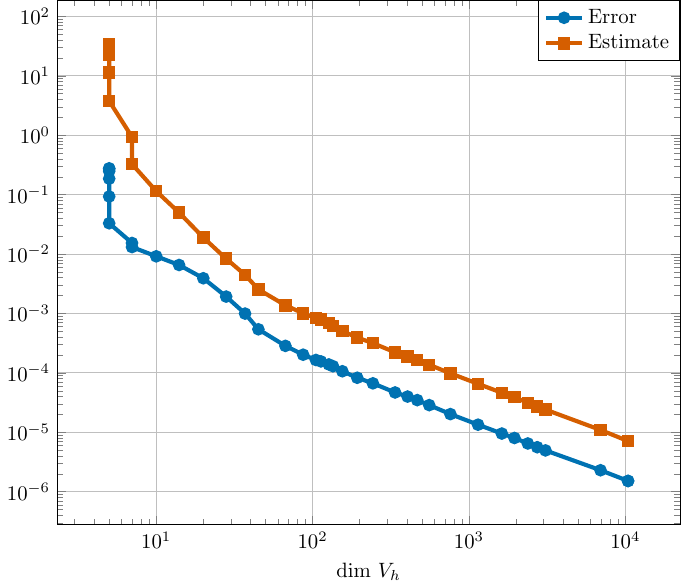}
    \caption{Error and Estimator}
    \label{fig:map}
  \end{subfigure}\hfil
  \begin{subfigure}{0.35\textwidth}
    \includegraphics[width=\linewidth]{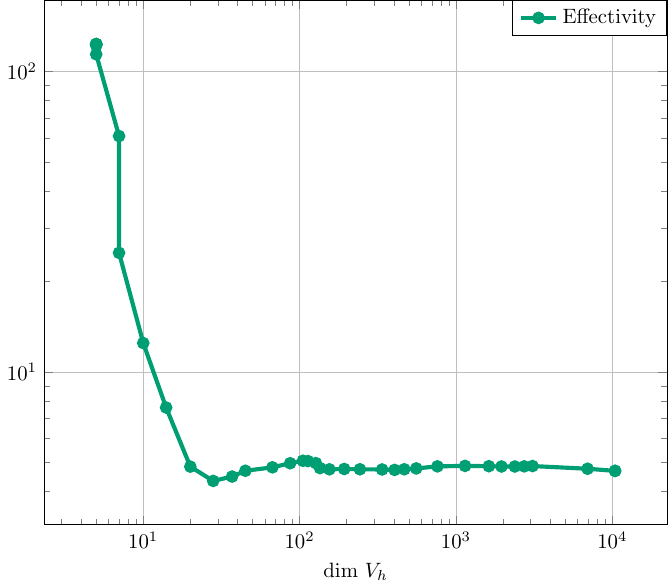}
    \caption{Effectivity}
    \label{fig:pcd}
  \end{subfigure}
    \begin{subfigure}{0.35\textwidth}
    \includegraphics[width=\linewidth]{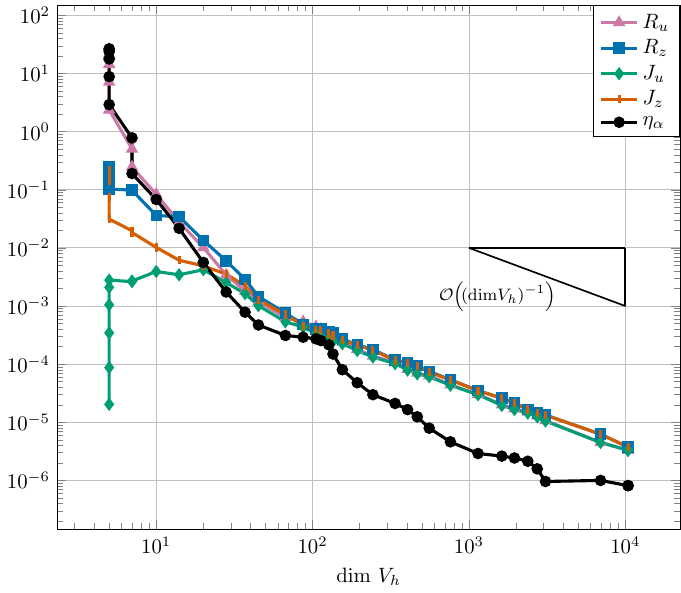}
    \caption{Estimator Components}
    \label{fig:gfindividual}
  \end{subfigure}\hfil
  \begin{subfigure}{0.35\textwidth}
    \includegraphics[width=\linewidth]{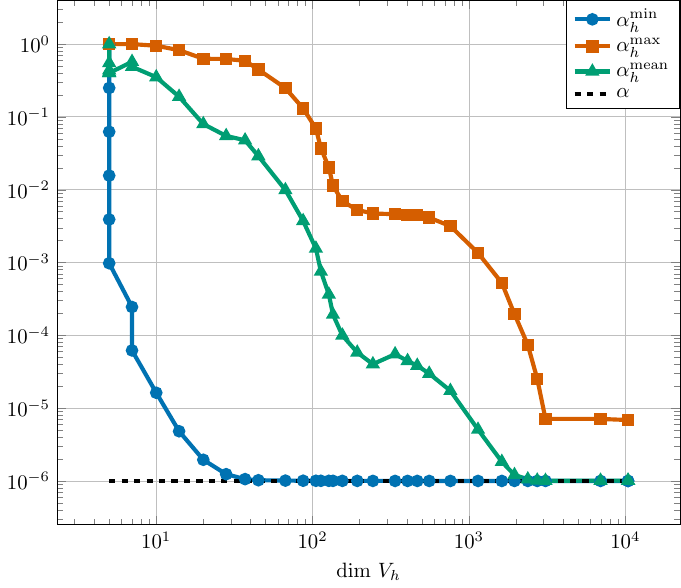}
    \caption{Parameter Ranges}
    \label{fig:pcd}
  \end{subfigure}
  \caption{Example \ref{bl} (boundary layer), \autoref{sec:fulladaptive}: convergence results for the fully adaptive scheme. $\rho = 0.5$, $\theta = 0.5$ and $\alpha = 10^{-6}$.}
  \label{fig:bl_full}
\end{figure}

\begin{figure}[h!]
  \centering
  \begin{subfigure}{\textwidth}
    \includegraphics[width=\linewidth]{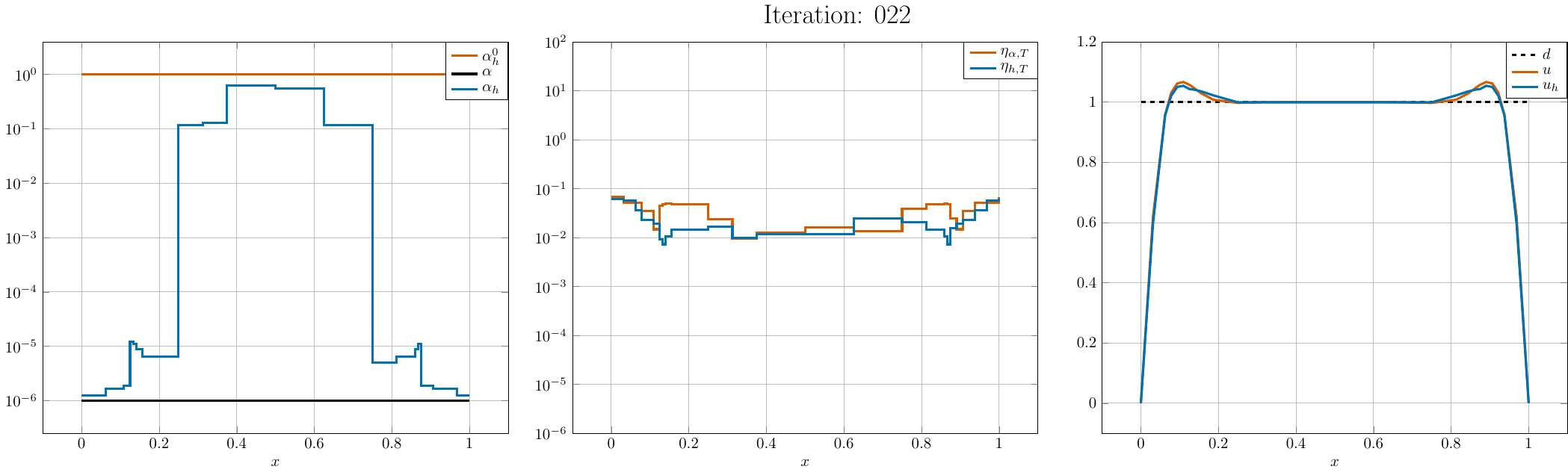}
    \label{fig:map}
  \end{subfigure}
  \begin{subfigure}{\textwidth}
    \includegraphics[width=\linewidth]{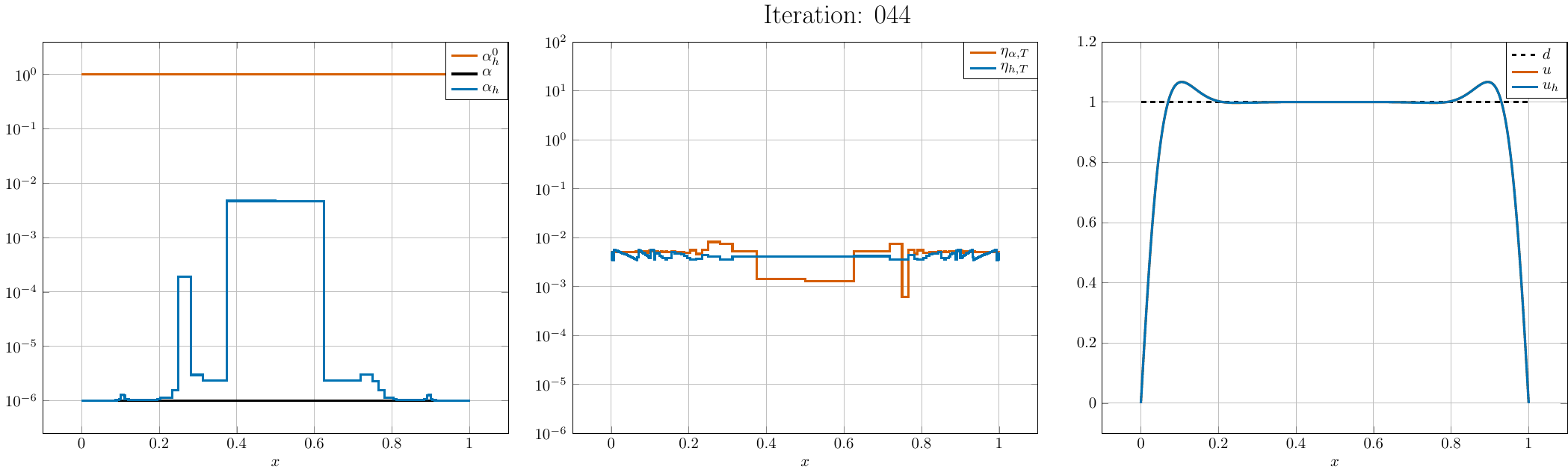}
    \label{}
  \end{subfigure}
    \begin{subfigure}{\textwidth}
    \includegraphics[width=\linewidth]{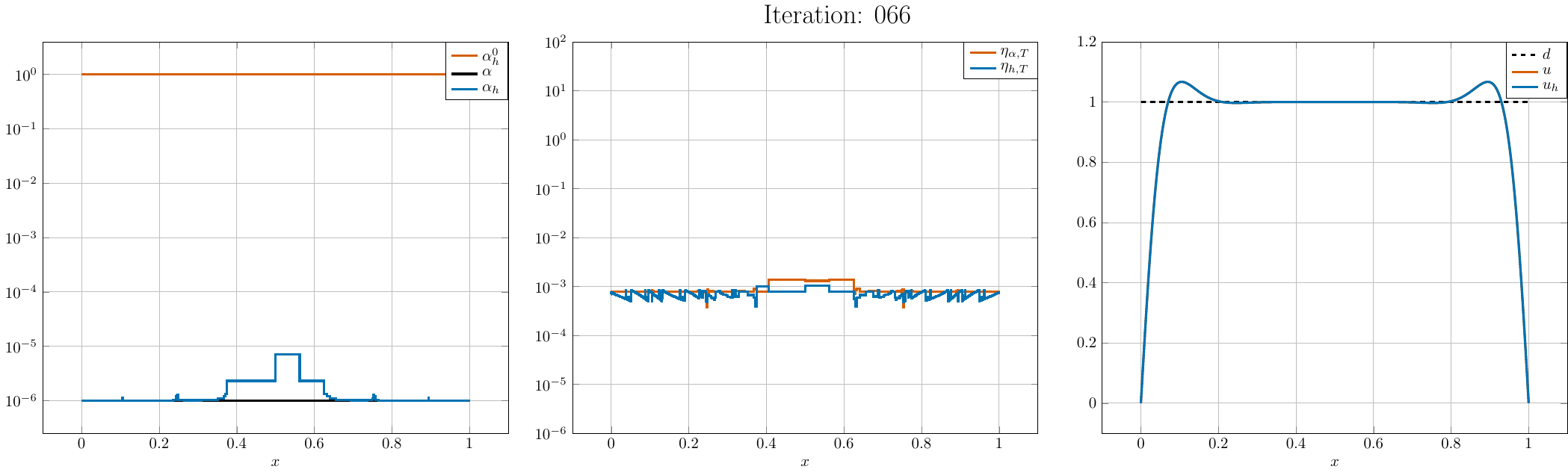}
    \label{fig:}
  \end{subfigure}
  \caption{Example \ref{bl} (boundary layer), \autoref{sec:fulladaptive}, fully adaptive scheme.. Plots of $\alpha_h$ (left), $\eta_\alpha, \eta_h$ (middle) and the approximate solution $u_h$ compared to $u$ and $d$ (right) over the domain at four different iterations. $\rho = 0.5$, $\theta = 0.5$ and $\alpha = 10^{-6}$.}
  \label{fig:bl_full_stills}
\end{figure}

We repeat the same experiments as in \autoref{sec:regadpt}, this time for the full adaptive algorithm. For both examples, an initial choice of $\alpha_h^0 = 1$, regularisation refinement fraction and mesh refinement ratio of $\rho, \theta = 0.5$. 

The results for example \ref{gaussian} (smooth target) can be seen in Figures \ref{fig:gaussian_full} and \ref{fig:gaussian_full_stills}.  The regularisation parameter was set to $\alpha = 10^{-8}$ and error tolerances of $\texttt{tol}_\alpha, \texttt{tol}_h = 10^{-6}$ were used.
\newpage

Convergence was reached after 59 iterations of Algorithm \ref{alg:1} with the final mesh having dimension $\text{dim}V_h \approx 5\times 10^4$. 
From examining the solution plots in Figure \ref{fig:gaussian_full_stills} it is evident that $\alpha_h$ decreases most in the region surrounding the Gaussian peak, as expected. However, in this experiment, the mesh is also being adaptively refined. The mesh is refined most in the vicinity of the Gaussian peak also. The method balances the regularisation and discretisation errors, which can be observed by looking at the estimator plots in Figure \ref{fig:gaussian_full_stills} (middle column). Both $\eta_{\alpha,T}$ and $\eta_{h,T}$ maintain similar orders of magnitude across the domain as the algorithm progresses. Examining Figure \ref{fig:gfindividual}, towards the end of the iterations, $\eta_\alpha$ decreases much faster than the components of $\eta_h$.

The results for example \ref{bl} can be seen in Figures \ref{fig:bl_full} and \ref{fig:bl_full_stills}. The regularisation parameter was set to $\alpha = 10^{-6}$ and error tolerances of and error tolerances of $\texttt{tol}_\alpha = 10^{-6}, \texttt{tol}_h = 10^{-5}$ were used. 
Convergence was reached after 70 iterations of algorithm \ref{alg:1} with the final mesh having dimension $\text{dim}V_h \approx 10^5$. The regularisation and mesh refinement both focus mainly on the area surrounding the boundary layer, as anticipated.

\subsection{A 2D Example}
We now look at a 2D example. We consider the square domain $\Omega = (0,1) \times (0,1) $. Let the spatial coordinates be denoted by $x = (x_1, x_2)$. 
\example[2D Example]{
\label{2D}
We define the piecewise-constant discontinuous target
\begin{equation} \label{eq:d2D}
  d({x}) = 
  \begin{cases}
      1, \threespace \text{ if }x \in \mathcal{C},\\
      0, \threespace \text{ otherwise, }
  \end{cases}
\end{equation}
where $\mathcal{C}$ is a disc centred at $(c_1, c_2)$ with radius $r$ defined by
\begin{equation*}
  \mathcal{C} := \big \{x \in \Omega \,\, \big |\,\, (x_1 - c_1)^2 + (x_2 - c_2)^2 \leq r^2     \big \}. 
\end{equation*}
This example does not have a closed form solution. For numerical experiments we choose $c_1,c_2= 0.5$ and radius $r=0.1$.
}

\begin{figure}[h!]
  \centering
  \begin{subfigure}[t]{0.32\textwidth}
    \includegraphics[width=\linewidth]{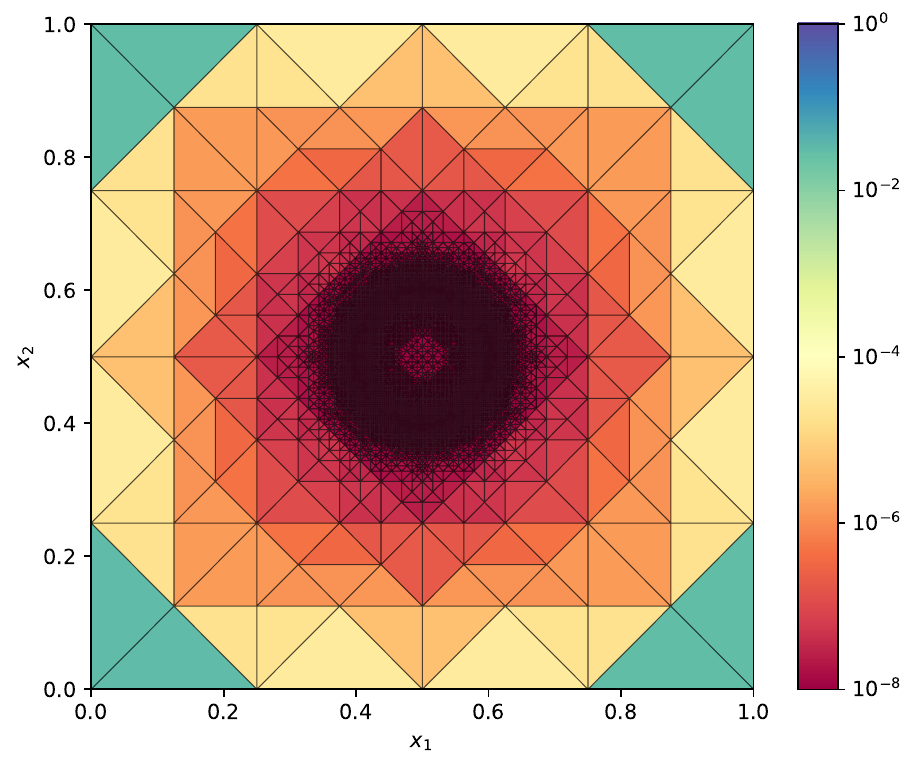}
    \caption{Colour plot of $\alpha_h$ on the mesh.}
    \label{fig:alphaonmesh}
  \end{subfigure}
  \begin{subfigure}[t]{0.32\textwidth}
    \includegraphics[width=\linewidth]{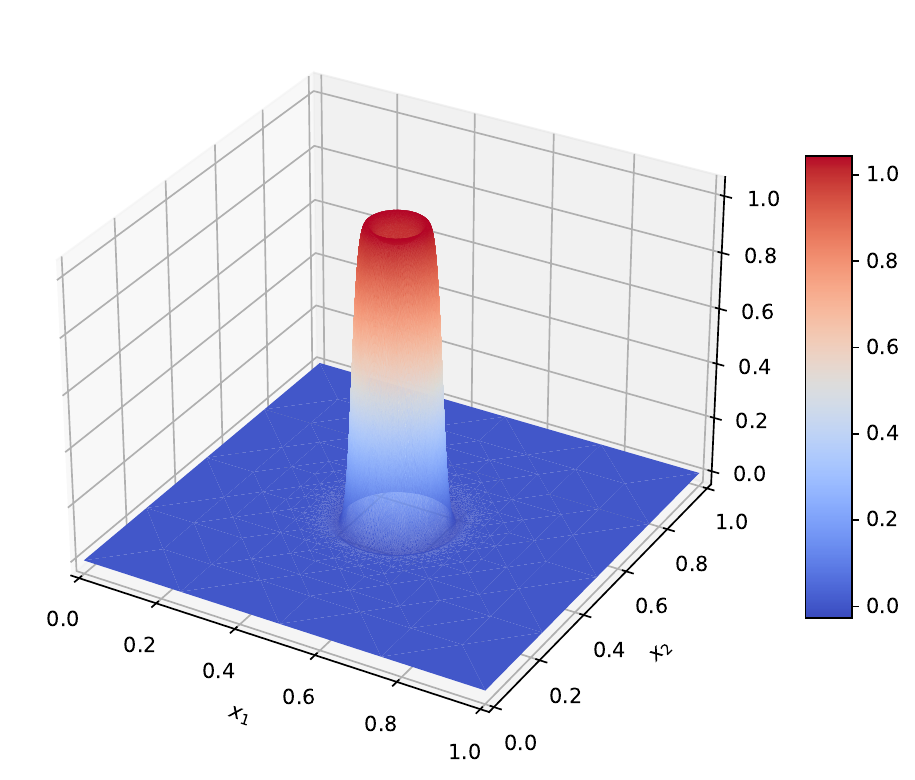}
    \caption{$u_h$}
    \label{fig:2Duh}
  \end{subfigure}
    \begin{subfigure}[t]{0.32\textwidth}
    \includegraphics[width=\linewidth]{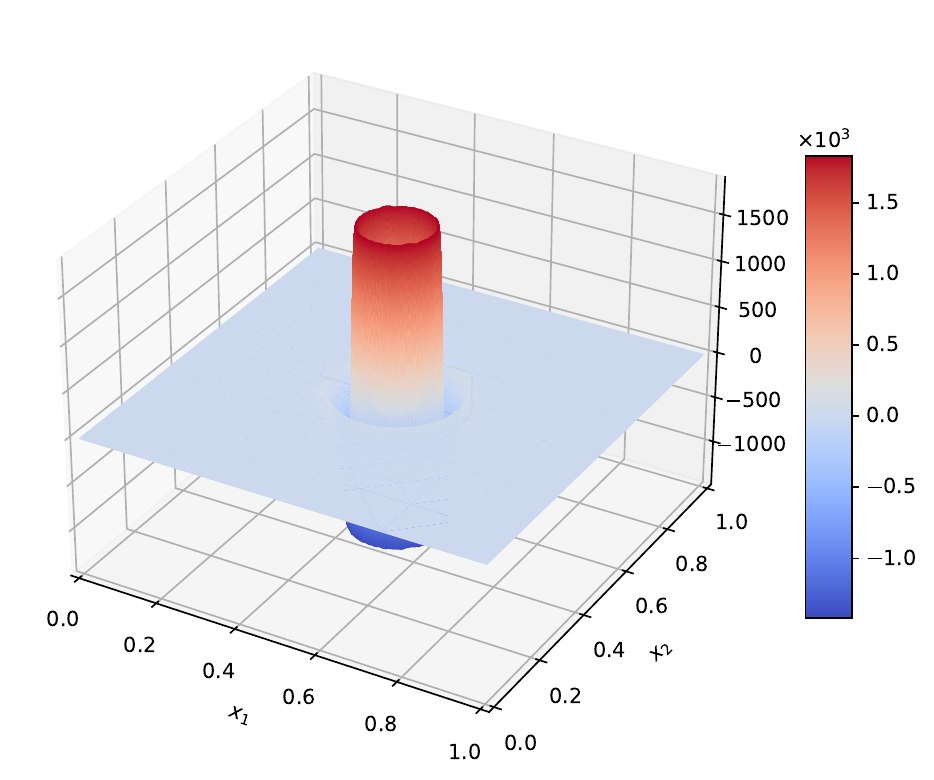}
    \caption{$f_h$}
    \label{fig:2Dfh}
  \end{subfigure}
  \hfil
  \begin{subfigure}[t]{0.32\textwidth}
    \includegraphics[width=\linewidth]{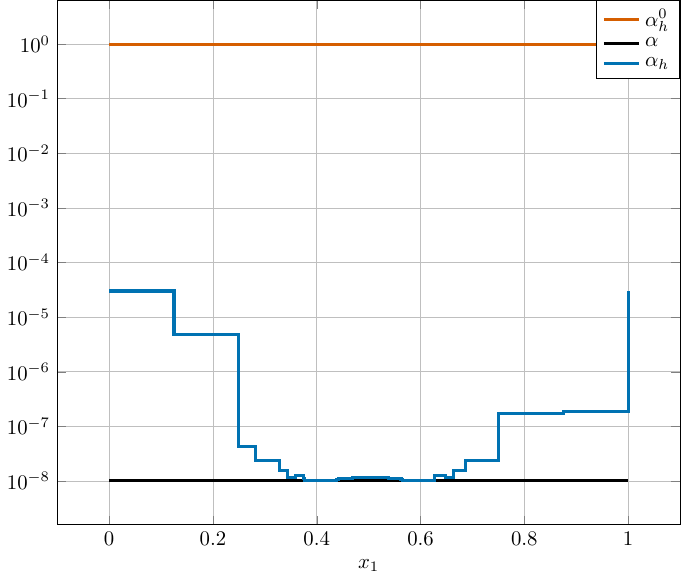}
    \caption{Slice of $\alpha_h$ along $x_2 = 0.5$}
    \label{fig:1Dalphah}
  \end{subfigure}
  \begin{subfigure}[t]{0.32\textwidth}
    \includegraphics[width=\linewidth]{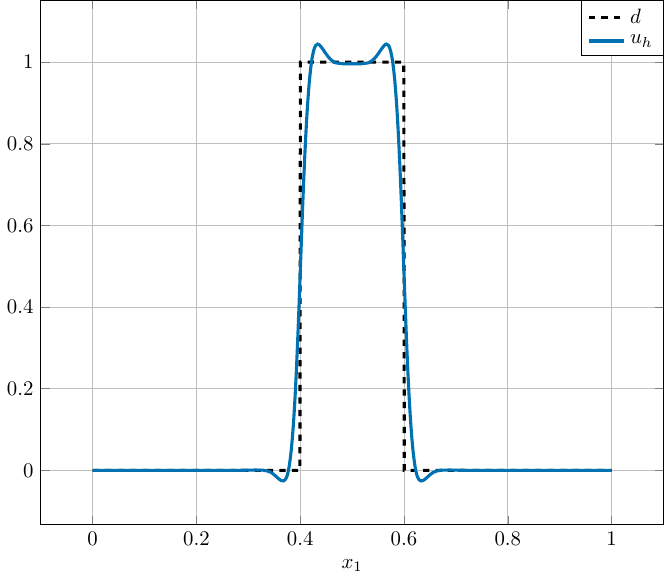}
    \caption{Slice of $u_h$ along $x_2 = 0.5$ compared to the target $d$ \eqref{eq:d2D}.}
    \label{}
  \end{subfigure}
    \begin{subfigure}[t]{0.32\textwidth}
    \includegraphics[width=\linewidth]{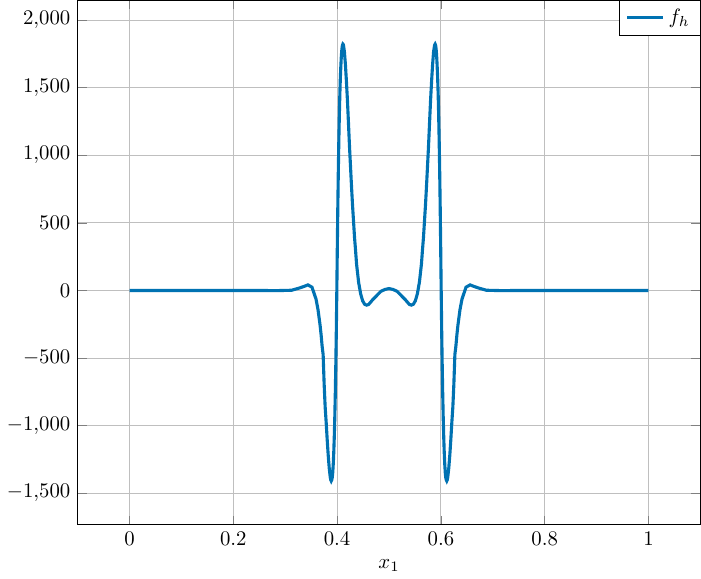}
    \caption{Slice of $f_h$ along $x_2 = 0.5$}
    \label{fig:1Dfh}
  \end{subfigure}
  \caption{Results for the full adaptive scheme on Example \ref{2D} (2D) with $\alpha = 10^{-8}$, final iteration. Algorithm parameters: $\rho = 0.5$, $\theta = 0.2$, $\texttt{tol}_\alpha = 10^{-5}$, $\texttt{tol}_h = 10^{-4}$. An animation of these results can be found in the supplementary material.}
  \label{fig:2Dresults}
\end{figure}

The full adaptive algorithm was run for Example \ref{2D}, the results of which can be seen in Figure \ref{fig:2Dresults}. This figure displays the output for the final iteration of the algorithm. A video of the results for each iteration can be found in the supplementary material. A regularisation parameter of $\alpha = 10^{-8}$ was chosen for this problem. The following algorithm parameters were used: 
a reduction rate of $\rho = 0.5$, a mesh refinement ratio of $\theta = 0.2$, error tolerances of $\texttt{tol}_\alpha = 10^{-5}$ and $\texttt{tol}_h = 10^{-4}$, and an initial choice of $\alpha_h^0 = 1$.

The algorithm converged after 59 iterations. The final mesh had a dimension of $\text{dim}V_h \approx 7 \times 10^5$. Figure \ref{fig:alphaonmesh} showcases the structure of the mesh. 
The majority of the mesh-refinement occurs in the vicinity of the discontinuity, as expected. This figure also showcases $\alpha_h$ on the mesh displayed on a log-scaled colour plot. In a similar vein, $\alpha_h$ takes on its smallest values in the vicinity of the discontinuity due to the sharp gradients in this region. The largest values of $\alpha_h$ are achieved at the corners of the domain, with $\max \alpha_h \sim \mathcal{O}(10^{-2})$. 

Figure \ref{fig:2Duh} displays the finite element approximation of the state variable $u_h$. We also showcase the reconstruction of the control variable $f_h$ in Figure \ref{fig:2Dfh}. 
1D slices of the $\alpha_h, u_h$ and $f_h$ along the line $x_2 = 0.5$ are given in Figures \ref{fig:1Dalphah}--\ref{fig:1Dfh}.

\subsection{Condition Number}
\begin{figure}[h!]
  \centering
  \begin{subfigure}[t]{0.32\textwidth}
    \includegraphics[width=\linewidth]{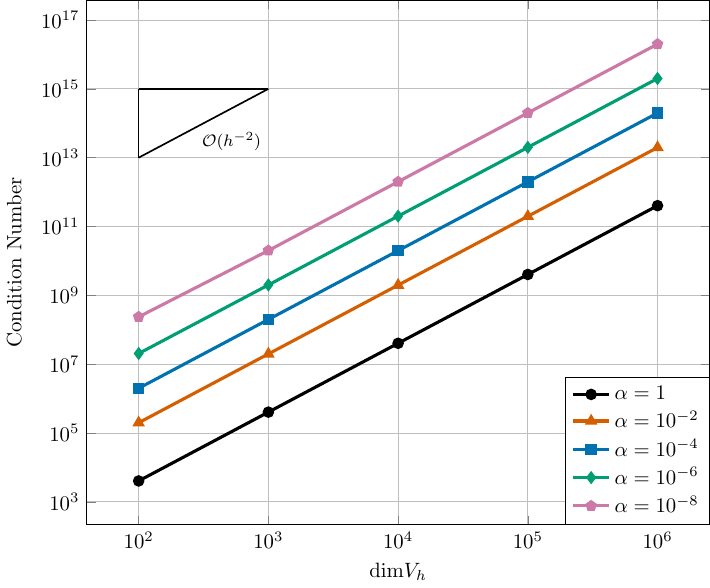}
    \caption{Condition Number}
    \label{fig:conda}
  \end{subfigure}\hfil
  \begin{subfigure}[t]{0.32\textwidth}
    \includegraphics[width=\linewidth]{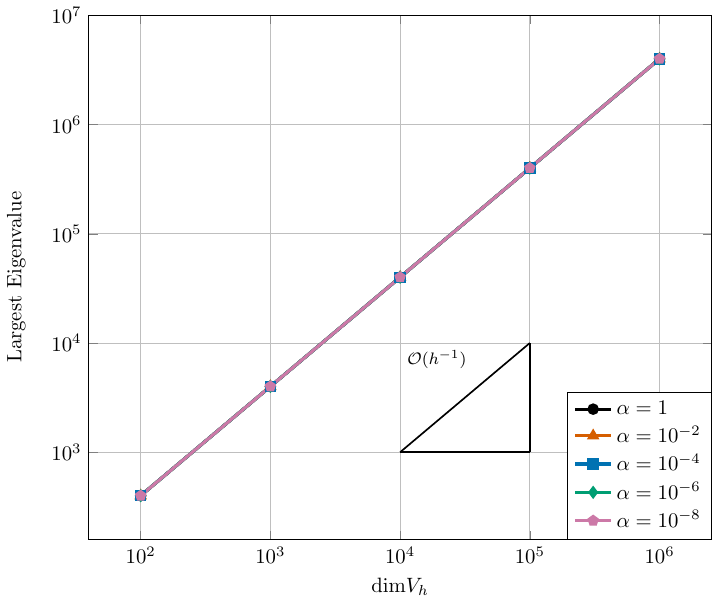}
    \caption{Largest Eigenvalue}
    \label{fig:condb}
  \end{subfigure}\hfil
    \begin{subfigure}[t]{0.32\textwidth}
    \includegraphics[width=\linewidth]{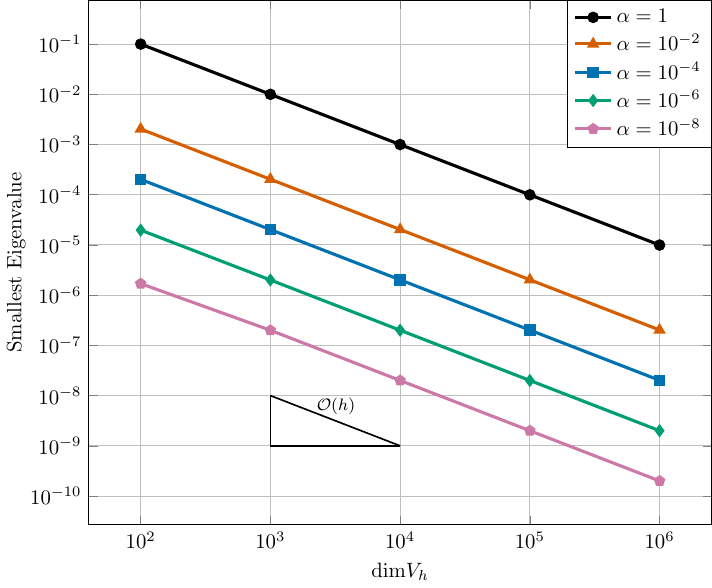}
    \caption{Smallest Eigenvalue}
    \label{fig:condc}
  \end{subfigure}\hfil
  \caption{The condition number, largest and smallest eigenvalues of the finite element system matrix obtained from solving \eqref{eq:bhform} with $\alpha_h = \alpha$. These were evaluated on a uniform mesh for different values of $\text{dim}V_h$ and for different values of $\alpha$.}
  \label{fig:cond}
\end{figure}

We conclude our experimentation with a study of the behaviour of the condition number of the finite element system matrix associated with solving \eqref{eq:bhform}. We first examine this for no adaptive regularisation $(\alpha_h = \alpha$) and investigate how the condition number scales with respect to $h$ and $\alpha$ in 1D. 
The condition number was estimated by the ratio of the largest and smallest eigenvalues of the system matrix. The eigenvalues were computed using SLEPc \cite{Hernandez:2005:SSF}.
The results are displayed in Figure \ref{fig:cond}. 
Figure \ref{fig:condc} displays the smallest eigenvalue as a function of the mesh size for a variety of $\alpha$'s. We can empirically deduce that the smallest eigenvalue scales like
\begin{equation*}
  \text{smallest eigenvalue} \sim \mathcal{O}\qp{h \alpha^{\frac{1}{2}}}.
\end{equation*}
Figure \ref{fig:condb} showcased the largest eigenvalue as a function of the mesh size for a variety of $\alpha$'s. The plots for the different choices of $\alpha$ coincide. Thus, we can deduce that when $\alpha < 1$ there is no relationship between $\alpha$ and the largest eigenvalue. The largest eigenvalue is only dependent on the mesh size and scales in the following way
\begin{equation*}
  \text{largest eigenvalue} \sim \mathcal{O}\qp{h^{-1}}.
\end{equation*}
Therefore, we can empirically deduce
that when $\alpha < 1$ the condition number scales in the following way,
\begin{equation*}
    \text{cond} \sim \mathcal{O}\qp{h^{-2} \alpha^{-\frac{1}{2}}}.
\end{equation*}
This results indicates that increasing the value of the regularisation parameter improves the conditioning of the linear system. This supports the use of  
adaptive regularisation in conjunction with adaptive mesh refinement.

\begin{figure}[h!]
  \centering
  \begin{subfigure}[t]{0.32\textwidth}
    \includegraphics[width=\linewidth]{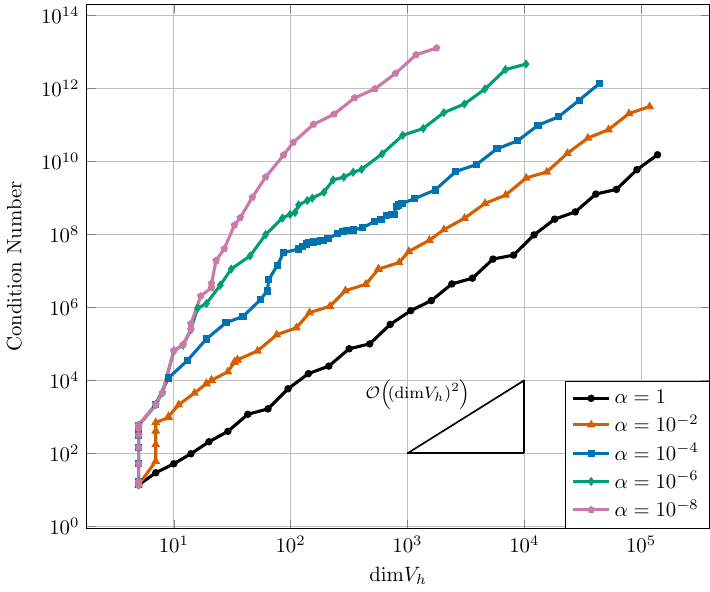}
    \caption{Condition Number}
    \label{fig:adaptcond}
  \end{subfigure}
  \hfil
  \begin{subfigure}[t]{0.32\textwidth}
    \includegraphics[width=\linewidth]{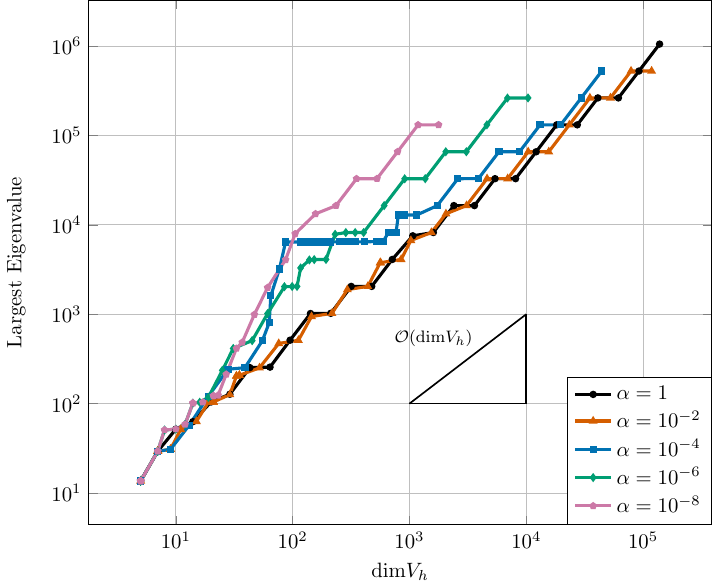}
    \caption{Largest Eigenvalue}
    \label{fig:adapteigmax}
  \end{subfigure}
  \hfil
    \begin{subfigure}[t]{0.32\textwidth}
    \includegraphics[width=\linewidth]{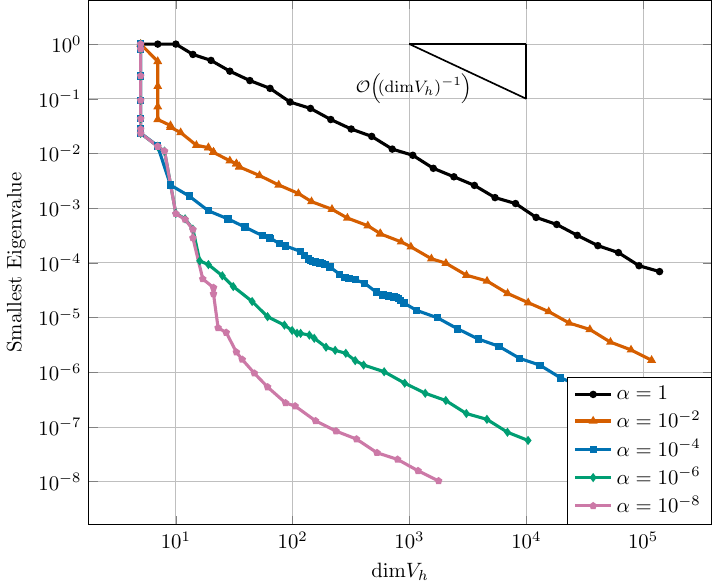}
    \caption{Smallest Eigenvalue}
    \label{fig:adapteigmin}
  \end{subfigure}
  \hfil
  \begin{subfigure}[t]{0.32\textwidth}
    \includegraphics[width=\linewidth]{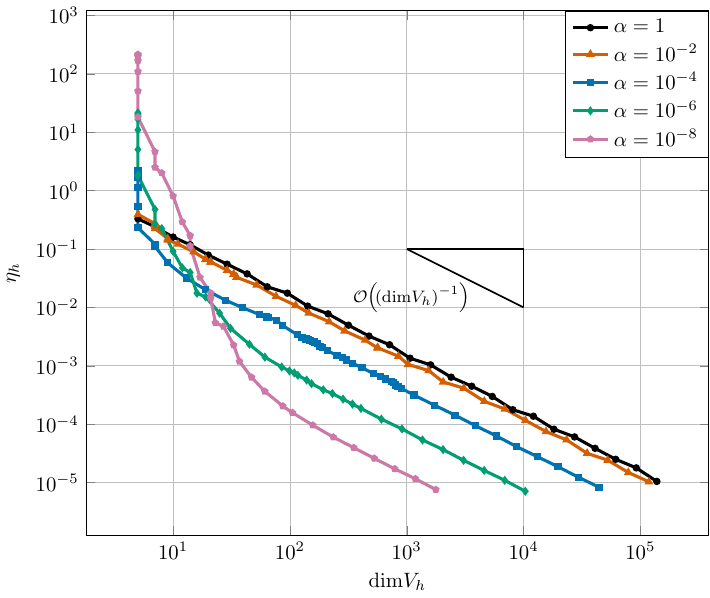}
    \caption{$\eta_h$}
    \label{fig:adaptetah}
  \end{subfigure}
  \hfil
    \begin{subfigure}[t]{0.32\textwidth}
    \includegraphics[width=\linewidth]{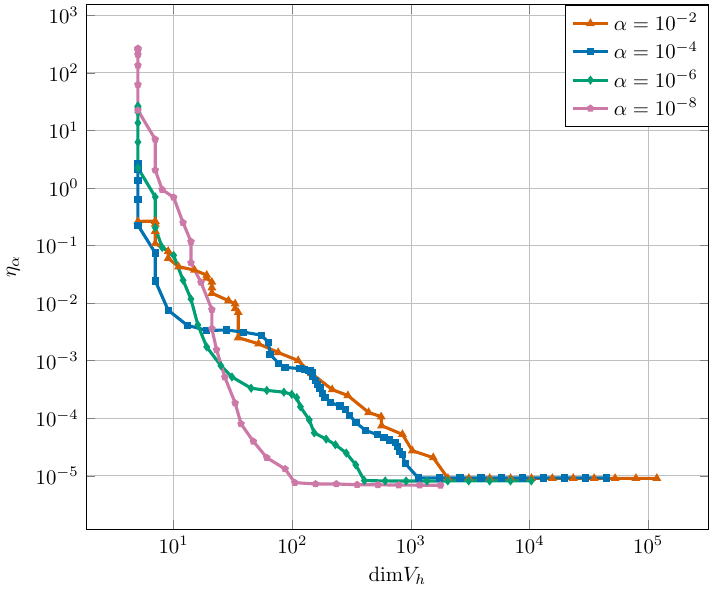}
    \caption{$\eta_\alpha$}
    \label{fig:adaptetaa}
  \end{subfigure}

  \caption{The results of applying the adaptive algorithm to Example \ref{bl} for a variety of choices of $\alpha$. The top row displays the condition number, largest and smallest eigenvalues of the finite element system matrix plotted against dim$V_h$. 
  The bottom row showcases the estimator $\eta_h$ and $\eta_\alpha$ against dim$V_h$. Algorithm parameters: $\texttt{tol}_{\alpha}, \texttt{tol}_{h} = 10^{-5}$, $\alpha_h^0 = 1$, $\theta = 0.5$, $\rho = 0.3$.}
  \label{fig:condadaptive}
\end{figure}

Lastly, we consider the condition number of the system
matrix for the adaptive scheme. We examine the boundary layer
example presented in Example \ref{bl} for different choices of
$\alpha$. The error tolerances are set to $\texttt{tol}_{\alpha}$
and $\texttt{tol}_{h} = 10^{-5}$, with an initial choice of
$\alpha_h^0 = 1$. A mesh refinement ratio of $\theta = 0.5$
and a regularisation reduction factor of $\rho = 0.3$ were
used. The results are presented in Figure \ref{fig:condadaptive}.

For the case $\alpha = 1$, this setup is equivalent to
solving the system with adaptive mesh refinement only, since the
initial choice satisfies $\alpha_h^0 = \alpha$. In this case,
the estimator $\eta_\alpha = 0$, and is therefore excluded from
Figure \ref{fig:adaptetaa}.  

We observe that once $\eta_\alpha < \texttt{tol}_\alpha$, the smallest eigenvalue begins to decrease
linearly with $\dim V_h$, as seen in the uniform refinement
case. Additionally, the largest eigenvalues no longer overlap,
indicating a deviation from the uniform behaviour. Finally, we note
that smaller values of $\alpha$ require less mesh
refinement. This occurs because as $\alpha$ decreases, the
problem becomes more singularly perturbed, leading to a more
localized region where the solution varies significantly.

\section{Conclusions}
In this work, we propose a novel approach to regularising
  PDE-constrained optimal control problems. These problems typically
  require a small regularisation parameter, making them equivalent to
  solving a singularly perturbed problem. Our method is based on an
  inconsistent approximation of the singular perturbation, allowing
  for a stable numerical approximation. Specifically, we replace the
  small parameter with one that varies spatially and can be chosen
  adaptively. We derive rigorous upper and lower a posteriori error
  bounds for a finite element scheme based on this inconsistent
  formulation and demonstrate the efficacy of the method through
  numerical examples, including smooth targets with local support,
  boundary layers and discontinuous targets. Our results highlight how
  this approach balances regularisation and discretisation errors.

  In this work, we restrict our analysis to a basic optimal
  control problem: regularisation using the $\leb{2}$-norm of the
  control, with the PDE constraint given by the Poisson equation. A
  natural extension would be to consider weaker regularisation norms,
  such as energy regularisation, as in
  \cite{langer2024adaptive,langer2022robust}. Additionally,
  investigating the case where the regularisation parameter varies
  spatially, $\alpha = \alpha(x)$, is of particular interest. In this
  setting, higher-order polynomial approximations of $\alpha_h$ could
  improve the convergence of the inconsistency term, motivated by
  Lemma \ref{lem:strang}.  
  
  Another avenue for future work is to
  explore coupling the choice of $\alpha_h$ to the mesh estimator
  $\eta_h$. While we opted to keep them independent to clarify their
  individual roles, an alternative approach could involve adapting
  $\alpha_h$ alongside mesh coarsening, analogous to techniques used
  in evolutionary problems \cite[cf.]{lakkis2012gradient}.

\section*{Acknowledgements}
JP is supported by a scholarship from the EPSRC Centre for Doctoral Training in Statistical Applied Mathematics at Bath (SAMBa), under the project EP/S022945/1. TP received support from the EPSRC programme grant EP/W026899/1 and was also
supported by the Leverhulme RPG-2021-238 and EPSRC grant EP/X030067/1.

\FloatBarrier
\printbibliography

\end{document}